\documentclass{article}

\makeatletter

\@addtoreset{equation}{section}
\makeatother

\usepackage{amsmath}
\usepackage{amssymb}
\usepackage{amsthm}
\usepackage{graphics}
\usepackage{epsfig}
\usepackage{fullpage}

\def\b{{\beta}}
\def\d{{\delta}}
\def\e{{\varepsilon}}
\def\k{{\kappa}}
\def\l{{\lambda}}

\def\pa{{\partial}}

\def\r{{\mathbb{R}}}

\newtheorem{theorem}{Theorem}

\newtheorem{lemma}{Lemma}
\newtheorem{proposition}{Proposition}

\newtheorem{remark}{Remark}

\title{Existence of traveling wave solutions in continuous OV models}
\author{Kota Ikeda\thanks{School of Interdisciplinary Mathematical Sciences, 
Meiji University, ikeda@meiji.ac.jp},
Toru Kan\thanks{Department of Mathematics, Osaka Metropolitan University},
Toshiyuki Ogawa\thanks{School of Interdisciplinary Mathematical Sciences, 
Meiji University}}
\date{}

\begin{document}
\maketitle

\begin{abstract}
 In traffic flow, self-organized wave propagation, which characterizes congestion, 
 has been reproduced in macroscopic and microscopic models. 
 Hydrodynamic models, a subset of macroscopic models, 
 can be derived from microscopic-level car-following models, 
 and the relationship between these models has been investigated. 
 However, most validations have relied on numerical methods and formal analyses; 
 therefore, analytical approaches are necessary to rigorously ensure their validity. 
 This study aims to investigate the relationship between macroscopic 
 and microscopic models based on the properties of the solutions corresponding to 
 congestion with sparse and dense waves. 
 Specifically, we demonstrate the existence of traveling wave solutions 
 in macroscopic models and investigate their properties.
\end{abstract}

\section{Introduction}\label{sec_Introduction}

Various aspects of traffic dynamics and congestion formation 
present challenges for mathematicians and physicists, 
drawing on more than 80 years of engineering experience.
In the early 1990s, traffic flow was recognized as a non-equilibrium system.
Empirical evidence indicates multiple dynamic phases 
in traffic flow and dynamic phase transitions.
Several mathematical models have been proposed to explain these empirical results, 
with some models qualitatively reproduce all known features of traffic flows, 
including localized and extended forms of congestion, 
self-organized propagation of stop-and-go waves,
and observed hysteresis effects.
These characteristics are criteria for good traffic models, as noted 
by Helbing \cite{helbing2001traffic}.
However, many of these models have only been validated using numerical 
techniques or formal analyses and have not yet been rigorously proven.

Mathematical models can be categorized into macroscopic and microscopic models.
These models are often interrelated through Taylor and mean-field approximations.
Payne \cite{payne1979critical} developed a macroscopic model based
on the compressible fluid equation and the dynamic velocity equation.
It was demonstrated in \cite{kuhne1984macroscopic} that 
the linear instability condition of the Payne model aligns precisely 
with those of well-known microscopic models, such as the car-following model 
or the optimal velocity model proposed by \cite{bando1994structure}.
However, the Payne model produces shock-like waves that compromise numerical robustness.
Hence, K\"uhne \cite{kuhne1987freeway} and Kerner and 
Konh{\"a}user \cite{kerner1993cluster} introduced models incorporating 
artificial viscosity terms to the Payne model.
In these studies, uniform flows were destabilized based on density,
and numerical calculations confirmed the stable formation of vehicle clusters.
Lee et al. \cite{lee2001macroscopic} attempted to derive a fluid-dynamic model
from a car-following model using a coarse-graining procedure to 
elucidate the relationship between these models.
The authors verified this through numerical simulations, demonstrating 
that the macroscopic model based on the mean-field method quantitatively approximates 
the microscopic model.

As noted by \cite{helbing2001traffic}, many macroscopic models,
including those mentioned above, can be expressed in the following general form:
\begin{equation}\label{eq_cOV}
 \left\{
 \begin{aligned}
  \dfrac{\pa \rho}{\pa t} + \dfrac{\pa (\rho v)}{\pa x} & = 0, \\
  \dfrac{\pa v}{\pa t} + v \dfrac{\pa v}{\pa x} 
  & = - \dfrac{1}{\rho} \dfrac{d (P (\rho))}{d x}
  + \k (\rho) \dfrac{\pa^2 v}{\pa x^2} + \dfrac{1}{\tau} (V (\rho^{-1}) - v),
 \end{aligned}
 \right.
\end{equation}
where $\rho = \rho (x, t)$ and $v = v (x, t)$ 
denote the density and velocity of the vehicles at $x$ and $t$, respectively.
The terms $\k, \tau, P$ and $V$ represent viscosity-like quantities, 
relaxation time, traffic pressure, and equilibrium velocity, respectively.
We neglected the diffusion effect of the density in the first equation and 
external forces, such as fluctuations, in (\ref{eq_cOV}).
We assume that $V$ and $\k$ depend on $\rho$, 
and $P (\rho) \equiv - V (\rho^{-1}) / 2 \tau$, where $\tau > 0$ is constant.
In specific cases, $V$ is often given by a sigmoid function
\begin{equation}\label{eq1_OV}
 V (u) \equiv V_0 [\tanh (\b (u - u_c)) + M],
\end{equation}
where $V_0, \b, u_c, M$ are positive constants. 
The function $V$ is known 
as the {\it optimal velocity function} \cite{bando1994structure}.
The density-dependent function $\k = \k (\rho)$ has been assumed to be 
$0$ \cite{payne1979critical},
$\k_0$ \cite{kuhne1984macroscopic, kuhne1987freeway},
$\k_0 / \rho$ \cite{kerner1993cluster},
and $1 / (6 \tau \rho^2)$ \cite{lee2001macroscopic},
where $\k_0$ denotes a positive constant.
We do not consider the case $\k (\rho) = 0$ in this paper.
The derivations and analyses of continuous models that are not included 
in (\ref{eq_cOV}) have also been conducted
(see for instance \cite{berg2000continuum, berg2001traveling}).

The microscopic optimal velocity model exhibits two typical types of 
collective motion depending on the density.
In the low-density region, the distance between any two neighboring vehicles
converges to a constant $t \to \infty$, known as {\it free flow}.
When the density is relatively high, the distance oscillates over time, 
known as {\it congestion} or {\it jamming}.
The transition between these two states occurs via a Hopf bifurcation, 
as shown in Fig.~12 of \cite{gasser2004bifurcation}.
The characteristics of the global bifurcation diagram 
indicate that a congested state has only one congested region, 
and all periodic solutions on any other bifurcation branch 
associated with multiple congested regions are unstable.
This implies that multiple congested regions merge as $t$ increases
and combine into a single lump.
These observations in the microscopic model are valid for the macroscopic model.
More precisely, the congestion phenomenon of vehicles in (\ref{eq_cOV}) 
can be considered as the dynamics of a {\it traveling wave solution},
which moves at a constant speed and forms a pulse shape
(see Figure~\ref{fig_congestion}).
If the initial state has a relatively high $\rho$, then 
all solutions transition to states with multiple pulses
after a short time, as long as numerical calculations are feasible.
Eventually, the multiple pulses merge into one.
Therefore analyzing 
the single-pulse traveling wave solution in (\ref{eq_cOV}) is crucial
to understand congestion phenomena in vehicles.

Generally, $\varphi_a$ represents the (partial) derivative of $\varphi$ 
with respect to $a$ for a function $\varphi$ that depends on variable $a$, 
i.e., $\varphi_a = \pa \varphi / \pa a$.
If $\varphi$ depends solely on one variable $a$, then 
we may use the symbol $\varphi'$, instead of $\varphi_a$.
We make the following assumptions for the optimal velocity function $V$ 
and viscosity coefficient $\k$ throughout the study.
\begin{itemize}
 \item[(A1)]
	    $V \in C^1 ([0, \infty)) \cap C^2 ((0, \infty))$.
	    $V' (u)$ is positive and bounded in $u > 0$,
	    and converges to $0$ as $u \to \infty$.
	    Moreover, there is a global maximum point $U_M \in (0, \infty)$
	    of $V' (u)$ such that $V'' (u) = 0$ attains a unique root at $u = U_M$.
 \item[(A2)]
            $\k \in C^1 ((0, \infty))$ and $\k > 0$ in $(0, \infty)$.
\end{itemize}
It follows from (A1) that $V' (u)$ has no local minimum points.
In some of our results, we additionally need the following assumption
for the diffusion coefficient $\k$
(see condition (\ref{eq1_cs})):
\begin{equation}\label{cd:nu}
 \int_1^\infty \frac{d\rho}{\rho^3 \k (\rho)} = \infty. \tag{H} 
\end{equation}
%


\begin{figure*}[h]
 \begin{center}
  \begin{tabular}{c}
   \includegraphics[width=12cm]{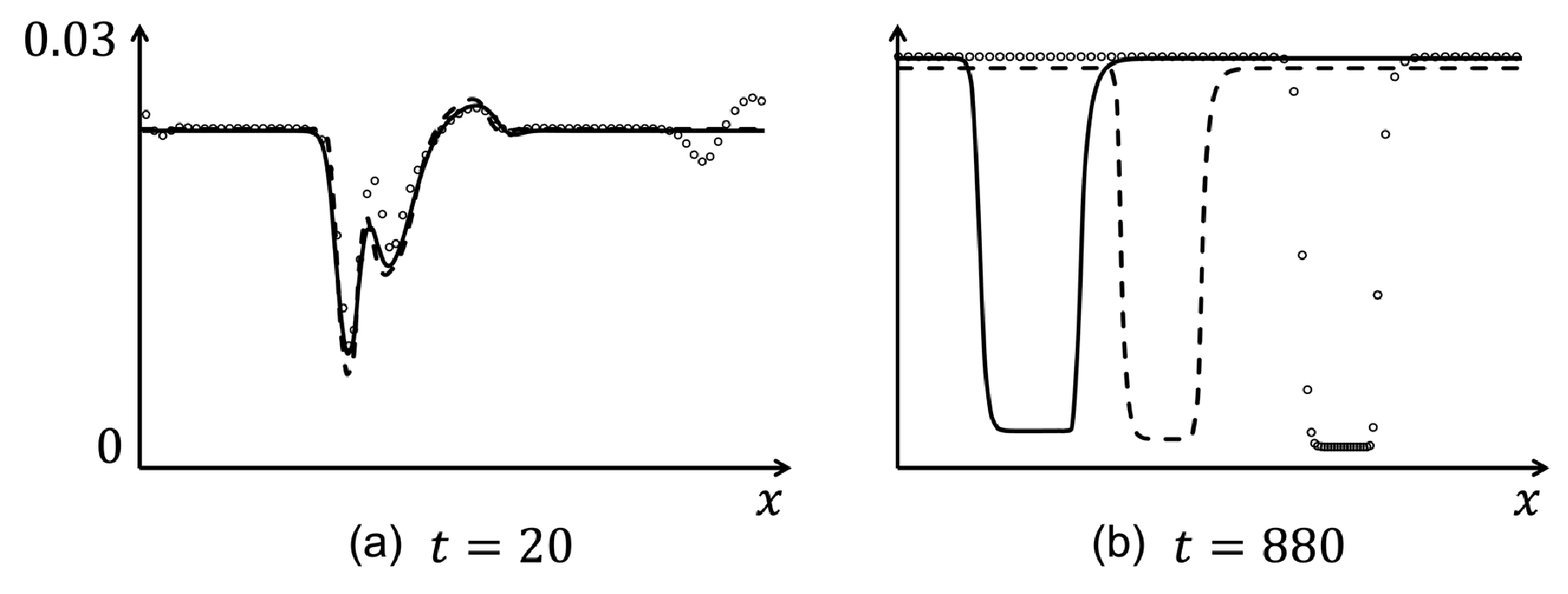}
  \end{tabular}
  \caption{Snapshots of solutions in (\ref{eq_cOV}) with the periodic boundary 
  condition and in the microscopic optimal velocity model at (a) $t = 20$ 
  and (b) $t = 880$.
  The optimal velocity function $V$ is given by (\ref{eq1_OV}).
  Solid and dashed lines represent the graphs of $v$ in (\ref{eq_cOV})
  under $1 / (6 \tau \rho^2)$ (Lee et al. model)
  and $\k (\rho) = \k_0$ (K\"uhne model), respectively.
  The unfilled circles denote the pairs of $(x_i, dx_i/dt)$ 
  in the microscopic optimal velocity model, 
  where $x_i$ represents the position of the $i$-th vehicle for $i = 1, \ldots, 77$.
  The numerical simulations were performed with parameters
  $L = 2.33$, $V_0 = 0.0168$, $M = 0.913$, 
  $u_c = 0.025$, $\beta = 89.7$, $\tau = 0.5$, and $\k_0 = 1/7500$,
  where $L$ denotes the length of the domain.
  The Fourier-spectral method was used in (\ref{eq_cOV})
  for space discretization with a truncation wave number of $100$, 
  while the time discretization employed 
  the classical Euler scheme with a time increment of $0.001$.  
  The classical fourth-order Runge-Kutta scheme 
  was used for the microscopic optimal velocity model 
  for time discretization with a time increment of $0.01$.
  }
  \label{fig_congestion}
 \end{center}
\end{figure*}


The traveling wave solution $(\rho (x, t), v (x, t)) = (\rho (z), v (z))$
in (\ref{eq_cOV}) is governed by 
\begin{equation}\label{eq1_TV0}
 \left\{
 \begin{aligned}
  c \rho_z + (\rho v)_z & = 0, \\
  c v_z + v v_z
  & = - \dfrac{1}{\rho} \dfrac{d (P (\rho))}{d z}
  + \k (\rho) v_{zz} + \dfrac{1}{\tau} (V (\rho^{-1}) - v),
 \end{aligned}
 \right.
\end{equation}
where $c$ is the wave speed and $z=x+ct$ is the moving coordinate.
As shown in Figure~\ref{fig_congestion}, the traveling wave solution to (\ref{eq_cOV})
approaches a constant outside the region where $v$ is relatively small, implying
the existence of a homoclinic orbit for a certain value of $c$.

We are interested in traveling wave solutions that connect constant steady states, 
imposing the condition:
\begin{equation}\label{bctw}
 \lim_{z \to \pm \infty} (\rho (z), v (z)) 
  = (\overline{\rho}_\pm, \overline{v}_\pm)
\end{equation}
for certain $(\overline{\rho}_\pm, \overline{v}_\pm)$, 
where $\overline{v}_\pm = V ((\overline{\rho}_\pm)^{-1})$.
If $\overline{\rho}_+ > \overline{\rho}_-$ 
(resp. $\overline{\rho}_+ < \overline{\rho}_-$),
the corresponding solution is a {\it traveling back solution}
(resp. {\it traveling front solution}).
If $\overline{\rho}_+ = \overline{\rho}_-$, it is
called a {\it traveling pulse solution}.
We are also interested in a {\it periodic traveling wave solution},
which satisfies $(\rho (z + Z), v (z + Z)) = (\rho (z), v (z))$
for some $Z > 0$ instead of \eqref{bctw}.

We obtain the following by integrating the first equation of (\ref{eq1_TV0}):
\begin{equation}\label{eq1_preserve_K}
 \rho (z) (v (z) + c) = K
\end{equation}
for any $z$, where $K$ is a constant.
We assume $K \neq 0$ because of our interest in obtaining 
nonconstant solutions of (\ref{eq1_TV0}).
We obtain the following by substituting 
$u \equiv \rho^{-1} = (v + c) / K$ and $P (\rho) = - V (u) / 2 \tau$ 
into the second equation of (\ref{eq1_TV0}):
\[
K^2 u u_z = \dfrac{1}{2 \tau} V' (u) u u_z
+ \k (u^{-1}) K u_{zz} + \dfrac{1}{\tau} (V (u) - K u + c).
\]
This is equivalent to the following dynamical system:
\begin{equation}\label{eq1_TV}
 \left\{
 \begin{aligned}
  u_z & = w, \\
  w_z & = g_1(u) f(u) + g_2 (u) h (u, \mu) w,
 \end{aligned}
 \right.
\end{equation}
where $\mu = 2 \tau K - 1$ and 
\[
\begin{aligned}
 & f (u) = \dfrac{K u - V (u) - c}{K},
 & & h (u, \mu) = f' (u) + \mu,
 & & g_1 (u) = \dfrac{1}{\tau \k (u^{-1})},
 & & g_2 (u) = \dfrac{u}{2 \tau \k (u^{-1})}.
\end{aligned}
\]
The traveling back/front, traveling pulse, and periodic solutions 
of (\ref{eq1_TV0}) correspond to a {\it heteroclinic orbit},
a {\it homoclinic orbit}, and a periodic orbit in (\ref{eq1_TV}), respectively.
Hence, they are also referred to as a traveling back/front solution,
a traveling pulse solution, and a periodic solution in (\ref{eq1_TV}).
Our goal is to prove the existence of these solutions.
We require that the solution $(u, w)$ satisfies $u (z) > 0$ 
throughout the study because $u = \rho^{-1}$ and $\rho$ must be positive.
We note that $c$ and $K$ are unknown constants, 
which must be determined such that (\ref{eq1_TV}) has appropriate orbits.
This task may be simplified by considering $\mu$ as a control parameter 
rather than addressing $(K, c)$ directly. 
Specifically, we seek a particular value of $\mu$ such that 
the desired traveling wave solutions exist for a given pair $(K, c)$.
The original problem can be then solved by determining $(K, c)$ 
such that $\mu = 2 \tau K-1$.
This study presents the first step toward addressing these issues.
Our results identify all $(K, c, \mu)$ such that a homoclinic 
or periodic orbit exists, demonstrating 
the existence of particular triplets and the 
nonexistence of others.

We provide several definitions and notations to state the main results.
Define $K_0 \equiv V'(0)$ and $K_M \equiv V'(U_M) > K_0$.
The assumption (A1) implies that the function $K u - V (u)$ has 
a unique local maximum point $u_M = u_M (K) \in (0,U_M)$ for $K \in (K_0, K_M)$.
Similarly, for $K \in (0, K_M)$, $K u - V (u)$ has a unique local minimum point
$u_m = u_m (K) \in (U_M, \infty)$.
Set $c_0 \equiv - V (0)$ and define $c_M = c_M (K)$,
$c_m = c_m (K)$, and $c_1 = c_1 (K)$ by
\[
c_M \equiv K u_M - V(u_M),
\quad c_m \equiv K u_m - V (u_m),
\quad c_1 \equiv \max \{ c_0, c_m \}. 
\]
It is elementary to show that some $K_1 \in (K_0, K_M)$ exists such that
\[
\begin{aligned}
 & c_m (K) < c_0 & & \mbox{if} \ 0 < K < K_1, \\
 & c_m (K) = c_0 & & \mbox{if} \ K = K_1, \\
 & c_m (K) > c_0 & & \mbox{if} \ K > K_1.
\end{aligned}
\]
After that, we define
\[
\begin{aligned}
 \mathcal{D}_1 
 & \equiv \{ (K, c) \ | \ K_0 < K < K_M, \ c_1 (K) < c < c_M(K) \}, \\
 \mathcal{D}_2 
 & \equiv \{ (K, c) \ | \ K_0 < K < K_1, \ c = c_0 \}
 \cup \{ (K, c) \ | \ 0 < K < K_1, \ c_m (K) < c < c_0 \}.
\end{aligned}
\]
If $(K, c) \in \mathcal{D}_1$, the function $f$ has three zeros, 
denoted by $u_0, u_1, u_2$ with $0 < u_1 < u_0 < u_2$.
Similarly, if $(K, c) \in \mathcal{D}_2$, $f$ has two zeroes, 
$u_0, u_2$, with $0 < u_0 < u_2$.
More precise properties of $f$ will be discussed in Lemma~\ref{lem:prf}.
Since we aim to find positive solutions in (\ref{eq1_TV}), 
we divide a region of $(K, c)$ into $\mathcal{D}_1$ and $\mathcal{D}_2$ 
based on the number of zeroes of $f$ for $u > 0$.

First, we study the existence of heteroclinic orbits in (\ref{eq1_TV})
that satisfy one of the following conditions:
\begin{equation}\label{eq:bcmi}
u (-\infty) = u_1
\quad u (\infty) = u_2,
\quad w (\pm \infty) = 0,
\quad w > 0,
\tag{HE1}
\end{equation}
\begin{equation}\label{eq:bcmd}
u (-\infty) = u_2,
\quad u (\infty) = u_1,
\quad w (\pm \infty) = 0,
\quad w < 0.
\tag{HE2}
\end{equation}
%


\begin{theorem}\label{thm:HECO}
 For $(K, c) \in \mathcal{D}_1$, there exists $\mu_b = \mu_b (K,c) \in \mathbb{R}$ 
 (resp. $\mu_f = \mu_f (K,c) \in \mathbb{R}$) such that (\ref{eq1_TV}) 
 with (\ref{eq:bcmi}) (resp. (\ref{eq:bcmd}))
 has a solution if and only if $\mu = \mu_b$ (resp. $\mu = \mu_f$).
\end{theorem}


Subsequently, we state the existence and nonexistence of homoclinic orbits
in (\ref{eq1_TV}), which satisfy
\begin{equation}\label{eq:bctmp0}\tag{HO}
 (u (\pm \infty), w (\pm \infty)) = (\overline{u},0),
  \quad (u(z), w(z)) \not\equiv (\overline{u},0),
\end{equation}
where $\overline{u}$ is a positive number satisfying $f(\overline{u})=0$.
The solution to (\ref{eq1_TV}) under condition (\ref{eq:bctmp0}) is called 
a traveling pulse solution.
As will be shown in Proposition~\ref{prop:BMBF}, 
we can find $c_* = c_* (K)$ such that 
$\mu_* \equiv \mu_b (K, c_* (K)) = \mu_f (K, c_* (K))$
under one of the following conditions:
\begin{equation}\label{eq1_cs}
 K \in (K_1, K_M)
  \quad \mbox{or} \quad \mbox{(H) and} \ K \in (K_0, K_1].
  \tag{C}
\end{equation}
This implies that (\ref{eq1_TV}) for $(c, \mu) = (c_*, \mu_*)$ 
has a {\it heteroclinic cycle} consisting of the equilibrium points
$(u_1,0)$ and $(u_2,0)$ and two heteroclinic orbits that join them.
After that, we divided $\mathcal{D}_1$ into
\[
\begin{aligned}
 \mathcal{D}_{1,1} & \equiv 
 \{ (K, c) \in \mathcal{D}_1 \ | \ c_*(K) < c < c_M (K) \}, \\
 \mathcal{D}_{1,2} & \equiv 
 \{ (K, c) \in \mathcal{D}_1 \ | \ c_1 (K) < c < c_*(K) \}, \\
 \mathcal{D}_{1,3} & \equiv \{ (K, c) \in \mathcal{D}_1 \ | \ c = c_*(K) \}.
\end{aligned}
\]
%


\begin{theorem}\label{thm:HOCO}
 The following statements hold:
 \begin{itemize}
  \item[(i)]
	    If $(K,c) \not \in \mathcal{D}_1 \cup \mathcal{D}_2$,
	    then (\ref{eq1_TV}) has no solution satisfying \eqref{eq:bctmp0}
	    for any $\overline{u}$ and $\mu \in \mathbb{R}$.
  \item[(ii)]
	     If $(K,c) \in \mathcal{D}_1 \cup \mathcal{D}_2$,
	     then (\ref{eq1_TV}) has no solution 
	     satisfying \eqref{eq:bctmp0} with $\overline{u} = u_0$
	     for any $\mu \in \mathbb{R}$.
  \item[(iii)]
	      Assume (\ref{eq1_cs}).
	      If $(K,c) \in \mathcal{D}_{1,2} \cup \mathcal{D}_{1,3}$ 
	      (resp. $(K,c) \in \mathcal{D}_{1,1} \cup \mathcal{D}_{1,3}$),
	      then (\ref{eq1_TV}) with \eqref{eq:bctmp0} has no solution 
	      for $\overline{u}=u_1$ 
	      (resp. $\overline{u}=u_2$) for any $\mu \in \mathbb{R}$.
  \item[(iv)]
	     Assume (\ref{eq1_cs}).
	     If $(K,c) \in \mathcal{D}_{1,1}$ 
	     (resp. $(K,c) \in \mathcal{D}_{1,2} \cup \mathcal{D}_2$),
	     then there exists $\mu^1_{pul} = \mu^1_{pul} (K,c)$ 
	     (resp. $\mu^2_{pul} = \mu^2_{pul}(K,c)$) 
	     such that (\ref{eq1_TV}) with (\ref{eq:bctmp0})
	     has a solution with $\overline{u}=u_1$ (resp. $\overline{u}=u_2$)
	     if and only if $\mu=\mu^1_{pul}$ (resp. $\mu=\mu^2_{pul}$).
 \end{itemize}
\end{theorem}


Finally, we discuss the existence or nonexistence 
of periodic orbits in (\ref{eq1_TV}).
Considering the Poincar\'{e} section $\{ w = 0 \}$,
we obtain a periodic solution $(u, w)$ to (\ref{eq1_TV}) 
satisfying the initial condition $(u (0), w (0)) = (q, 0)$.


\begin{theorem}\label{thm:ptws} 
 Assume (\ref{eq1_cs}).
 Let $q$ satisfy
 \begin{equation}\label{cd:rq}
  q \in 
   \left\{
    \begin{aligned}
     & (u_1, u_0) && \mbox{if } \ (K,c) \in \mathcal{D}_{1,1} 
     \cup \mathcal{D}_{1,3}, \\
     & (u_0, u_2) && \mbox{if } \ (K,c) \in \mathcal{D}_{1,2} 
     \cup \mathcal{D}_2.
    \end{aligned}
   \right.
 \end{equation}
 Then there exists $\mu_{per} = \mu_{per} (K, c, q) > 0$ 
 such that (\ref{eq1_TV}) has a periodic solution 
 with the initial condition $(u (0), w (0)) = (q, 0)$ for $\mu = \mu_{per}$.
 Moreover, if $(K, c) \not \in \mathcal{D}_1 \cup \mathcal{D}_2$
 or $\mu \neq \mu_{per}$, there does not exist any periodic solution to (\ref{eq1_TV}).
\end{theorem}


As previously mentioned, (\ref{eq1_TV}) exhibits a heteroclinic cycle 
for $(K, c) \in \mathcal{D}_{1,3}$ under assumption~(\ref{eq1_cs}).
As discussed in \cite{kokubu1988homoclinic}, a global bifurcation
can occur in systems with heteroclinic cycles, 
where a homoclinic orbit emerges from a heteroclinic cycle.
Sections~\ref{sec_bifurcation} and \ref{Numerical results}
demonstrate that such a bifurcation occurs in (\ref{eq1_TV}).
In our case, a homoclinic orbit obtained in Theorem~\ref{thm:HOCO} (iv)
bifurcates from the heteroclinic cycle.
However, this study does not use the global bifurcation results
established by \cite{kokubu1988homoclinic}.

All theorems are proven using phase-plane analysis. 
The key of the proofs is the monotonicity of 
solution trajectories with respect to $\mu$
(see Lemmas~\ref{lem:mwmu}, \ref{lem:mwmc} below).
The shooting method facilitates us to directly study the behavior 
of the solution initiated from the equilibrium points.
Such methods are widely used to demonstrate the existence of traveling wave solutions
(see \cite{aronson2006nonlinear}).

The remainder of this paper is organized as follows.
Section~\ref{sec:pre} discusses the investigation of the basic properties of 
(\ref{eq_cOV}) and (\ref{eq1_TV}).
Additionally, the properties of $f$ are described.
In Section~\ref{sec:TF}, we consider the existence of the heteroclinic orbits 
in (\ref{eq1_TV}) and Theorem~\ref{thm:HECO}. 
Section~\ref{sec_mubmuf} discusses the properties of 
$\mu_b$ and $\mu_f$ as given in Theorem~\ref{thm:HECO} 
(see Proposition~\ref{prop:BMBF}).
Sections~\ref{sec:TPS} and \ref{sec_Structure of periodic traveling wave solutions} 
address the existence of homoclinic and periodic orbits
in (\ref{eq1_TV}) as stated in Theorems~\ref{thm:HOCO} and \ref{thm:ptws}.
Section~\ref{sec_bifurcation} discusses the parameter dependence 
of heteroclinic, homoclinic, and periodic orbits. 
Additionally, we study the bifurcations that occur when these solutions appear.
Section~\ref{Numerical results} presents a numerical investigation of the relationship 
among $K, c, \mu$ in the presence of traveling wave solutions, revealing 
the global bifurcation structure of (\ref{eq1_TV}) using AUTO \cite{doedel2007auto}.
One of our results, shown in Figure~\ref{c-K}, 
is qualitatively equivalent to that in Figure~4 of \cite{lee2004steady},
where the authors categorized the parameter space $(K, c)$
based on the topological structure of the flow diagram
and highlighted the existence of a heteroclinic cycle without using AUTO.


\section{Preliminaries}\label{sec:pre}

\subsection{Basic properties of (\ref{eq_cOV})}

To begin with, we consider the local existence and uniqueness 
of a solution in (\ref{eq_cOV}).
If we consider (\ref{eq_cOV}) in the whole line $S \equiv \r$,
we impose
\begin{equation}\label{eq2_bcR}
 \lim_{x \to \pm \infty} (\rho (x, t), v (x, t)) 
  = (\overline{\rho}_\pm, \overline{v}_\pm).
\end{equation}
If we are concerned with the bounded interval $S \equiv (0, L)$,
we impose the periodic boundary conditions for $\rho$ and $v$.
Let $C^{k + \alpha} (S)$ for integer $k \ge 0$ and $0 < \alpha < 1$
be the set of $k$th-differentiable functions $\phi : S \to \r$
whose $k$th-derivatives are bounded and $\alpha$-H\"older continuous \cite{MR1814364}.
We find a unique solution $(\rho, v) \in C^{1+\alpha} (S) \times C^{2+\alpha} (S)$
locally in time.
The following proposition is standard so that we omit the details of the proof
(see \cite{MR1625845} and \cite{MR1329547}).


\begin{proposition}\label{thm_local_existence}
 Give $(\rho_0, v_0) \in C^{1+\alpha} (S) \times C^{2+\alpha} (S)$
 and assume that $T > 0$ is sufficiently small.
 Then the problem (\ref{eq_cOV}) with the initial condition
 $(\rho (z, 0), v (z, 0)) = (\rho_0 (z), v_0 (z))$ has a unique solution 
 $(\rho, v)$ in $C^{1+\alpha} (S) \times C^{2+\alpha} (S)$ in $(0, T)$. 
\end{proposition}


As described in Introduction, congestion in the microscopic 
optimal velocity model proposed in \cite{bando1994structure}
arises via Hopf bifurcation associated with the destabilization 
of the free flow \cite{gasser2004bifurcation}.
This fact strongly implies the instability 
of a constant stationary solution $(\rho_*, v_*)$ in (\ref{eq_cOV}), 
where $v_* = V (\rho_*^{-1})$.
To examine the $\rho_*$-dependency of the stability,
we set $(\rho, v) = (\rho_*, v_*) + (\phi, \psi) e^{\l t} e^{i k x}$
and study the linearized eigenvalue problem of (\ref{eq_cOV}), given by
\begin{equation}\label{eq2_ep_cOV}
 \left\{
 \begin{aligned}
  \l \phi + i k (\rho_* \psi + v_* \phi) & = 0, \\
  \l \psi + i k v_* \psi 
  & = - \dfrac{i k}{2 \tau \rho_*^3} V' (\rho_*^{-1}) \phi
  - \k (\rho_*) k^2 \psi
  - \dfrac{1}{\tau} \left(\dfrac{1}{\rho_*^2} V' (\rho_*^{-1}) \phi + \psi \right).
 \end{aligned}
 \right.
\end{equation}
Then we obtain the characteristic equation 
\[
(\l + i k v_*) \left( \l + i k v_* + \dfrac{1}{\tau} + \k (\rho_*) k^2 \right)
- \dfrac{i k}{\tau \rho_*}
\left( \dfrac{i k }{2 \rho_*} + 1 \right) V' (\rho_*^{-1}) = 0.
\]
The eigenvalues $\l = \l_{\pm} (k)$ for each $k \in \r$ 
are explicitly given by 
\[
\begin{aligned}
 \l_{\pm} (k) 
 & = - i k v_* + \dfrac{1}{2}
 \left( - \dfrac{1}{\tau} - \k (\rho_*) k^2 \pm \sqrt{z (k)} \right), \\
 z (k) & \equiv \left( \dfrac{1}{\tau} + \k (\rho_*) k^2 \right)^2
 + \dfrac{4 i k}{\tau \rho_*} 
 \left( \dfrac{i k}{2 \rho_*} + 1 \right) V' (\rho_*^{-1}). 
\end{aligned}
\]
Obviously, $\mbox{Re} \l_- (k) < 0$ for any $\rho_*$ and $k$,
where $\mbox{Re} \l$ stands for the real part of complex number $\l$.
To determine the stability of the stationary solution, it is sufficient 
to consider the sign of the real part of $\l_+ (k)$ for each $\rho_*$ and $k$.


\begin{lemma}\label{lemma_sign_mur}
 If $1 > 2 \tau V' (\rho_*^{-1})$, then
 there are $k_1 < 0$ and $k_2 > 0$ such that $\mbox{Re} \l_+ (k) > 0$
 in $k_1 < k < k_2$ except for $k = 0$
 and $\mbox{Re} \l_+ (k) < 0$ in $k < k_1, k_2 < k$.
 On the other hand, if $1 \le 2 \tau V' (\rho_*^{-1})$, 
 $\mbox{Re} \l_+ (k) < 0$ in any $k \neq 0$.
\end{lemma}


\begin{proof}
 Define $\l_r (k) = \mbox{Re} \l_+ (k)$.
 It is clear that $\l_r$ is smooth in $k \in \r$.
 It is easy to obtain $\l_r (0) = \l_r' (0) = 0$,
 and $\l_r'' (0) = - V' (\rho_*^{-1}) ( 1 - 2 \tau V' (\rho_*^{-1})) / \rho_*^2$.
 Hence, $1 - 2 \tau V' (\rho_*^{-1})$ determines the sign of $\l_r'' (0)$
 unless it is not equal to $0$.
 If $1 = 2 \tau V' (\rho_*^{-1})$,
 direct calculations imply $\l_r'' (0) = \l_r^{(3)} (0) = 0$ 
 and $\l_r^{(4)} (0) = - 24 V' (\rho_*^{-1}) \k (\rho_*) \tau / \rho_*^2 < 0$.
 Moreover, we easily calculate 
 \begin{equation}\label{eq2_lim}
  \lim_{k \to \pm \infty} \l_r (k) 
   = - \dfrac{V' (\rho_*^{-1})}{2 \tau \rho_*^2 \k (\rho_*)} < 0.
 \end{equation}

 Next we assume that there exist $k \neq 0$ and $a \in \r$ such that 
 $\l_+ (k) = i (a - k v_*)$, and calculate $\l_r' (k)$.
 Since 
 \[
 \sqrt{z (k)} = 2 i a + \dfrac{1}{\tau} + \k (\rho_*) k^2,
 \]
 we have 
 \[ 
 a^2 = \dfrac{k^2}{2 \tau \rho_*^2} V' (\rho_*^{-1}),
 \quad a \left( \dfrac{1}{\tau} + \k (\rho_*) k^2 \right)
 = \dfrac{k}{\tau \rho_*} V' (\rho_*^{-1}),
 \]
 from which we have 
 \[
 V' (\rho_*^{-1}) = \dfrac{\tau}{2} 
 \left( \dfrac{1}{\tau} + \k (\rho_*) k^2 \right)^2.
 \]
 Then we find
 \[
 a = \dfrac{k}{2 \rho_*} \left( \dfrac{1}{\tau} + \k (\rho_*) k^2 \right),
 \quad \sqrt{z (k)} 
 = \left( \dfrac{i k}{\rho_*} + 1 \right)
 \left( \dfrac{1}{\tau} + \k (\rho_*) k^2 \right).
 \]
 Direct calculations give us 
 \[
 \begin{aligned}
  z' (k) 
  & = 4 \k (\rho_*) k \left( \dfrac{1}{\tau} + \k (\rho_*) k^2 \right)
  + \dfrac{4 i}{\tau \rho_*} 
  \left( \dfrac{i k }{\rho_*} + 1 \right) V' (\rho_*^{-1}) \\
  & = \left( \dfrac{1}{\tau} + \k (\rho_*) k^2 \right)
  \left( 4 \k (\rho_*) k 
  - \dfrac{2 k}{\rho_*^2} \left( \dfrac{1}{\tau} + \k (\rho_*) k^2 \right) \right)
  + \dfrac{2 i}{\rho_*} \left( \dfrac{1}{\tau} + \k (\rho_*) k^2 \right)^2
 \end{aligned}
 \]
 and 
 \[
\begin{aligned}
 & \l_+' (k) + i v_* + \k (\rho_*) k
 = \dfrac{z' (k)}{4 \sqrt{z (k)}}
 = \dfrac{2 \k (\rho_*) k 
 - \dfrac{k}{\rho_*^2} \left( \dfrac{1}{\tau} + \k (\rho_*) k^2 \right) 
 + \dfrac{i}{\rho_*} \left( \dfrac{1}{\tau} + \k (\rho_*) k^2 \right)}
 {2 \left( \dfrac{i k}{\rho_*} + 1 \right)}.
\end{aligned}
 \]
 Picking up the real part of the above equality, we obtain
 \begin{equation}\label{eq_dmu_at_mu0}
  \l_r' (k) = - \dfrac{k^3 \k (\rho_*)}{k^2 + \rho_*^2},
 \end{equation}
 which shows that $\l_r' (k) < 0$ if $k > 0$,
 while $\l_r' (k) > 0$ if $k < 0$.

 Combining the facts above, we show Lemma~\ref{lemma_sign_mur}.
 Consider the case that $1 \le 2 \tau V' (\rho_*^{-1})$
 and assume that there exists $k > 0$ such that $\l_r (k) > 0$.
 Then there must be some $\tilde k > 0$ such that 
 $\l_r (\tilde k) = 0$ and $\l_r' (\tilde k) \ge 0$, which contradicts
 the sign of $\l_r' (k)$.
 On the other hand, if $1 > 2 \tau V' (\rho_*^{-1})$,
 then $\l_r'' (0) > 0$, (\ref{eq2_lim}) and (\ref{eq_dmu_at_mu0}) imply 
 the statement of Lemma~\ref{lemma_sign_mur} by the same argument as above.
\end{proof}


\subsection{Properties of $f$ and flows near equilibria}

We first summarize the sign and zeros of $f$.


\begin{lemma}\label{lem:prf}
 The following statements hold.
 \begin{itemize}
  \item[(i)]
	    If $(K,c) \in \mathcal{D}_1$, then $f$ has three zeros 
	    $u_0 = u_0 (K, c)$, $u_1 = u_1 (K, c)$, $u_2 = u_2(K, c)$ 
	    with $0 < u_1 < u_0 < u_2$ and satisfies
	    \begin{equation}\label{fbs}
	     \left\{
	      \begin{aligned}
	       & f(u) > 0 \ \mbox{ for } \ u \in (u_1, u_0) \cup (u_2, \infty), \\
	       & f(u) < 0 \ \mbox{ for } \ u \in (0, u_1) \cup (u_0,u_2), \\
	       & f'(u_1) > 0, \ f'(u_2)>0, \ f'(u_0) < 0.
	      \end{aligned}
	     \right.
	    \end{equation}
  \item[(ii)]
	     If $(K,c) \in \mathcal{D}_2$, then $f$ has two zeros 
	     $u_0 = u_0 (K, c)$, $u_2 = u_2 (K, c)$ with $0 < u_0 < u_2$ 
	     and satisfies
	     \[
	     \left\{
	     \begin{aligned}
	      & f(u) > 0 \ \mbox{ for } \ u \in (0,u_0) \cup (u_2,\infty), \\
	      & f(u) < 0 \ \mbox{ for } \ u \in (u_0,u_2), \\
	      & f'(u_2) > 0, \ f'(u_0) < 0.
	     \end{aligned}
	     \right.
	     \]
	     Moreover,
	     \[
	     \begin{aligned}
	      & f(0) = 0, \ f'(0)>0 \ \mbox{ if } 
	      \ K_0<K<K_1 \ \mbox{ and } \ c=c_0, \\
	      & f (0) > 0 \ \mbox{ if } \ 0<K<K_1 \ \mbox{ and } \ c_m<c<c_0.
	     \end{aligned}
	     \]
  \item[(iii)]
	     $u_0$, $u_1$ and $u_2$ are of class $C^1$ and satisfy
	     \begin{equation}\label{luacm}
	      u_0 (K,c), \ u_1 (K,c) \to u_M (K)
	       \ \mbox{ as } \ c \to c_M (K),
	     \end{equation}
	     \begin{equation}\label{lubcm}
	      u_0 (K, c), \ u_2 (K, c) \to u_m (K)
	      \ \mbox{ as } \ c \to c_m (K).
	     \end{equation}
 \end{itemize}
\end{lemma}


We next study the flows of (\ref{eq1_TV}) 
near the equilibria $(u_0,0)$, $(u_1,0)$ and $(u_2,0)$.
Let $F(u, w)$ and $J (u)$ be the two-dimensional vector field associated 
with (\ref{eq1_TV}) and the Jacobian matrix of $F$ at $(u, 0)$, respectively. 
They are explicitly given by 
\begin{equation}\label{def:FJ}
 F(u,w) \equiv 
  \begin{pmatrix}
   w \\
   g_1 (u) f (u) +g_2 (u) h (u, \mu) w
  \end{pmatrix}
  ,
  \quad J (u) \equiv 
  \begin{pmatrix}
   0 & 1 \\
   g_1 (u) f'(u) & g_2 (u) h (u, \mu)
  \end{pmatrix}
  .
\end{equation}
The eigenvalues $\l_\pm (u)$ of $J (u)$ are explicitly given by
\[
\l_\pm (u) = \frac{1}{2} 
\left( g_2 (u) h (u,\mu) \pm \sqrt{g_2 (u)^2 h (u,\mu)^2 
+ 4 g_1 (u) f' (u)} \right).
\]
It is elementary to verify the following lemma by direct calculation.


\begin{lemma}\label{lem:evus}
 The following hold.
 \begin{itemize}
  \item[(i)]
	    For $(K,c) \in \mathcal{D}_1$ 
	    (resp. $(K,c) \in \mathcal{D}_1 \cup \mathcal{D}_2$),
	    it holds that at $u = u_1$ (resp. $u = u_2$)
	    \[
	    \l_+ (u) > 0, \quad \l_- (u) < 0, 
	    \quad (\l_+)_\mu (u) > 0, 
	    \quad (\l_-)_\mu (u) > 0.
	    \]
  \item[(ii)]
	     Assume $(K,c) \in \mathcal{D}_1 \cup \mathcal{D}_2$.
	     If $\mu \neq - f'(u_0)$, then $\mbox{Re} \l_\pm (u_0)$ 
	     are either all positive or all negative,
	     while if $\mu = - f'(u_0)$, then $\l_\pm (u_0) = \pm i \omega_0$,
	     where $\omega_0 \equiv \sqrt{- g_1 (u_0) f'(u_0)} > 0$.
 \end{itemize}
\end{lemma}


\subsection{Useful lemmas}

We give some simple lemmas to be used throughout the subsequent sections.


\begin{lemma}\label{lem:Wpes0}
 Let $A = A(s)$ and $B = B(s)$ be continuous integrable functions 
 on a bounded interval $(s_1, s_2)$,
 and let $W \in C^1 ((s_1, s_2)) \cap C ([s_1,s_2))$ be a solution of 
 \[
 W' = A (s) W + B (s) \quad \mbox{in } (s_1, s_2). 
 \]
 Then $W$ can be extended continuously up to $s = s_2$.
 Moreover, if $B>0$ in $(s_1, s_2)$ and $W(s_1) \ge 0$,  
 then $W>0$ on $(s_1, s_2]$.
\end{lemma}


\begin{proof}
 Solving the differential equation above, we have
 \[
 W (s) = \exp \left( \int_{s_1}^s A(\tau)d\tau \right) W(s_1) 
 + \int_{s_1}^s \exp \left( \int_\tau^s A(\zeta)d\zeta \right) B(\tau) d \tau
 \]
 for $s \in [s_1,s_2)$.
 The lemma follows immediately from the above equality.
\end{proof}


\begin{lemma}\label{lem:Wpes}
 Let $A$ and $B$ be continuous functions on a bounded interval $[s_1, s_2]$. 
 Assume that $W \in C^1((s_1, s_2)) \cap C([s_1, s_2])$ satisfies
 \[
 \left( W^2\right)' =A(s) +B(s)W \quad \mbox{in } (s_1,s_2).
 \]
 Then there hold
 \[
 \begin{aligned}
  e^{-s_2}W(s_2)^2 -e^{-s_1}W(s_1)^2
  & \le \int_{s_1}^{s_2} e^{-s} \left( A(s) +\frac{1}{4}B(s)^2 \right) ds, \\
  e^{s_2}W(s_2)^2 -e^{s_1}W(s_1)^2
  & \ge \int_{s_1}^{s_2} e^{s} \left( A(s) -\frac{1}{4}B(s)^2 \right) ds.  
 \end{aligned}
 \]
\end{lemma}


\begin{proof}
 By the differential equation above 
 and the inequality $|BW| \le B^2/4 +W^2$, we have
 \[
 \left( W^2\right)' -W^2 \le A(s) +\frac{1}{4}B(s)^2,
 \quad \left( W^2\right)' +W^2 \ge A(s) -\frac{1}{4} B(s)^2,
 \]
 which concludes the lemma by Gronwall's inequality.
\end{proof}


\section{Existence of traveling back and front solutions}\label{sec:TF}

The goals of this section are to find monotone traveling back and front solutions
and to prove Theorem~\ref{thm:HECO}.
In order to obtain desired heteroclinic orbits,
we consider stable and unstable manifolds emanating from $(u_1,0)$ and $(u_2,0)$.
Let $(K,c) \in \mathcal{D}_1$.
Lemma~\ref{lem:evus} shows that for each $j = 1, 2$, 
there is an orbit $\{ (u^{\rm s}_j (z), w^{\rm s}_j (z)) \ | \ z \in \r \}$ 
(resp. $\{ (u^{\rm u}_j (z),w^{\rm u}_j (z)) \ | \ z \in \r \}$)
which lies on the stable (resp. unstable) manifold of the equilibrium
$(u_j, 0)$ and satisfies
\[
\begin{aligned}
 & (u^{\rm s}_1 (z), w^{\rm s}_1 (z)) \in S_-, 
 & & (u^{\rm s}_2 (z), w^{\rm s}_2 (z)) \in S_+,
 & & (u^{\rm u}_1 (-z),w^{\rm u}_1 (-z)) \in S_+, 
 & & (u^{\rm u}_2 (-z),w^{\rm u}_2 (-z)) \in S_-
\end{aligned}
\]
for large $z$, where $S_- \equiv  \{ (u,w) \ | \ u_1 < u < u_2, \ w < 0 \}$ 
and $S_+ \equiv \{ (u,w) \ | \ u_1 < u < u_2, \ w > 0 \}$.
Hence we define $z^{\rm s}_j$ and $z^{\rm u}_j$ by 
\begin{equation}\label{zsud}
 \begin{aligned}
  & z^{\rm s}_1 \equiv 
  \inf \{ z_0 \ | \ \{(u^{\rm s}_1 (z), w^{\rm s}_1(z)) \ | \ z > z_0 \} 
  \subset S_- \} \in [-\infty,\infty), \\
  & z^{\rm u}_1 \equiv 
  \sup \{ z_0 \ | \  \{(u^{\rm u}_1 (z), w^{\rm u}_1(z)) \ | \ z < z_0 \} 
  \subset S_+ \} \in (-\infty,\infty], \\
  & z^{\rm s}_2 \equiv 
  \inf \{ z_0 \ | \  \{(u^{\rm s}_2 (z), w^{\rm s}_2(z)) \ | \ z > z_0 \} 
  \subset S_+ \} \in [-\infty,\infty), \\
  &z^{\rm u}_2 \equiv 
  \sup \{ z_0 \ | \  \{(u^{\rm u}_2 (z), w^{\rm u}_2 (z)) \ | \ z < z_0 \} 
  \subset S_- \} \in (-\infty,\infty].
 \end{aligned}
\end{equation}
It follows from Lemma~\ref{lem:prf} that 
\begin{equation}\label{obtep}
 \begin{aligned}
  & (u^{\rm s}_1(z^{\rm s}_1),w^{\rm s}_1(z^{\rm s}_1)) 
  \in \{ (u, 0) \ | \ u_0 \le u \le u_2 \} 
  \cup \{ (u_2, w) \ | \ w < 0 \}, \\
  & (u^{\rm u}_1(z^{\rm u}_1),w^{\rm u}_1(z^{\rm u}_1)) 
  \in \{ (u, 0) \ | \ u_0 \le u \le u_2 \} \cup \{ (u_2, w) \ | \ w > 0 \}, \\
  & (u^{\rm s}_2(z^{\rm s}_2),w^{\rm s}_2(z^{\rm s}_2)) 
  \in \{ (u, 0) \ | \ u_1 \le u \le u_0 \} \cup \{ (u_1, w) \ | \ w > 0 \}, \\
  & (u^{\rm u}_2(z^{\rm u}_2),w^{\rm u}_2(z^{\rm u}_2)) 
  \in \{ (u, 0) \ | \ u_1 \le u \le u_0 \} \cup \{ (u_1, w) \ | \ w < 0 \}.
 \end{aligned}
\end{equation}
As far as $u_z \neq 0$, we see from the inverse function theorem that 
each of orbits
$\{ (u^{\rm s}_j (z), w^{\rm s}_j (z)) \ | \ z > z^{\rm s}_j \}$
and $\{ (u^{\rm u}_j (z), w^{\rm u}_j (z) ) \ | \ z < z^{\rm u}_j \}$
is expressed as the graph of a function of $u$.
More precisely, there are $u^\pm_j \in \r$
and functions $w^\pm_j = w^\pm_j (u)$ such that 
\begin{equation}\label{eq3_Opmj}
 \begin{aligned}
  \mathcal{O}^{+}_1 
  & \equiv \{ (u^{\rm u}_1 (z), w^{\rm u}_1 (z)) \ | \ z > z^{\rm u}_1 \}
  = \{ (u,w_1^+(u)) \ | \ u_1 < u < u_1^+ \}, \\
  \mathcal{O}^{-}_1 
  & \equiv \{ (u^{\rm s}_1 (z), w^{\rm s}_1 (z)) \ | \ z < z^{\rm s}_1 \}
  = \{ (u, w_1^-(u)) \ | \ u_1 < u < u_1^- \}, \\
  \mathcal{O}^{+}_2 
  & \equiv \{ (u^{\rm s}_2 (z), w^{\rm s}_2 (z) ) \ | \ z < z^{\rm s}_2 \}
  = \{ (u,w_2^+(u)) \ | \ u_2^+ < u < u_2 \}, \\
  \mathcal{O}^{-}_2 
  & \equiv \{ (u^{\rm s}_2 (z), w^{\rm s}_2 (z) ) \ | \ z > z^{\rm u}_2 \}
  = \{ (u,w_2^-(u)) \ | \ u_2^- < u < u_2 \}.
 \end{aligned} 
\end{equation}
To emphasize the dependency of the parameters, 
we may write $u^\pm_j = u^\pm_j (K,c,\mu)$ and $w^\pm_j (u) = w^\pm_j (u; K,c,\mu)$.
In the case of $u^\pm_j \neq u_1, u_0, u_2$, we see that 
$u^\pm_j$ is $C^1$ with respect to $(K,c,\mu)$ by the implicit function theorem.
By \eqref{obtep}, we have $u_0 \le u_1^\pm \le u_2$ and $u_1 \le u_2^\pm \le u_0$.

We see from (\ref{eq1_TV}) that $w^\pm_j$ is a solution of the equation
\begin{equation}\label{eq:wueq}
 w_u = \frac{g_1 (u) f (u)}{w} + g_2 (u) h (u,\mu),
\end{equation}
which is also written as
\begin{equation}\label{eq:wueq2}
 \frac{1}{2} (w^2)_u = g_1 (u) f (u) + g_2 (u) h (u, \mu) w.
\end{equation}
By Lemma~\ref{lem:evus}, 
we infer that $w^\pm_j$ can be extended smoothly up to $u = u_j$ and 
\begin{equation}\label{eq:wuic}
 \begin{aligned}
  & w_1^\pm = 0, 
  \quad (w_1^+)_u = \l_+ (u_1),
  \quad (w_1^-)_u = \l_- (u_1)
  & & \mbox{at } u = u_1, \\
  & w_2^\pm = 0,
  \quad (w_2^+)_u = \l_- (u_2),
  \quad (w_2^-)_u = \l_+ (u_2)
  & & \mbox{at } u=u_2.
 \end{aligned}
\end{equation}
We also infer that $w^\pm_j$ can be extended continuously up to $u=u_j^\pm$ and
\begin{equation}\label{wzep}
\begin{aligned}
 & w^\pm_1=0 \ \mbox{ at } \ u=u_1^\pm \ \mbox{ if } \ u_0 \le u_1^\pm<u_2, \\
 & w^\pm_2=0 \ \mbox{ at } \ u=u_2^\pm \ \mbox{ if } \ u_1<u_2^\pm \le u_0. 
\end{aligned}
\end{equation}
Moreover, $w^\pm_j$ are continuously differentiable with respect to $(K,\mu,c)$.
Also, $u^\pm_2$ and $w^\pm_2$ are defined 
for $(K, c) \in \mathcal{D}_2$ by setting $u_1 = 0$.

We show the monotonicity of $w_j^\pm$ with respect to $\mu$.


\begin{lemma}\label{lem:mwmu} 
 For $(K,c) \in \mathcal{D}_1$, there holds $(w^\pm_1)_\mu >0$ 
 in $u \in (u_1, u^\pm_1)$.
 Similarly, for $(K,c) \in \mathcal{D}_1 \cup \mathcal{D}_2$,
 there holds $(w^\pm_2)_\mu < 0$ in $u \in (u^\pm_2,u_2)$.
\end{lemma}


\begin{proof}
 We put $W\equiv (w^+_1)_\mu$.
 By \eqref{eq:wuic} and Lemma~\ref{lem:evus},
 we have $W = 0$, $W_u > 0$ at $u=u_1$.
 Hence we can pick a point $\hat u \in (u_1,u^+_1)$ 
 such that $W>0$ on $(u_1, \hat u]$.
 Differentiating \eqref{eq:wueq} with respect to $\mu$,
 we see that $W$ satisfies
 \[
 W_u = - \frac{g_1(u) f(u)}{(w^+_1)^2} W + g_2 (u) 
 \]
 in $(u_1, u^+_1)$.
 By applying Lemma~\ref{lem:Wpes0}, $W > 0$ in $(\hat u, u^+_1)$.
 Thus we conclude that $(w^+_1)_\mu > 0$ in $(u_1, u^+_1)$.
 By a similar argument, we obtain $(w^-_1)_\mu > 0$ and $(w^\pm_2)_\mu < 0$.
 Therefore the lemma follows.
\end{proof}


\begin{lemma}\label{lem:bwmi}
 Let $(K,c) \in \mathcal{D}_1$ and $a \in (u_1,u_2)$.
 If $\mu$ (resp. $-\mu$) is large enough,
 then $u_1^+>a$ and $u_2^-<a$ (resp. $u_1^->a$ and $u_2^+<a$).
 Moreover, there hold
 \begin{equation}\label{vwpmi}
  \begin{aligned}
   & w^\pm_1(a;K,c,\mu) \to \pm \infty 
   \ \mbox{ as } \ \mu \to \pm \infty, \\
   & w^\pm_2(a;K,c,\mu) \to \pm \infty
   \ \mbox{ as } \ \mu \to \mp \infty. 
  \end{aligned}
 \end{equation}
\end{lemma}


\begin{proof}
 To estimate $w^+_1$,
 we first integrate \eqref{eq:wueq}.
 Then, by Lemma~\ref{lem:prf}, we have
 \[
 w^+_1 (u) \ge \int_{u_1}^u g_2 (s) h (s,\mu) ds 
 \]
 for $u \in [u_1,u_0]$.
 Since the right-hand side goes to $\infty$ as $\mu \to \infty$, we find
 \begin{equation}\label{vwpmi0}
  w^+_1(u;K,c,\mu) \to \infty 
 \end{equation}
 locally uniformly for $u \in (u_1,u_0]$ as $\mu \to \infty$.

 Next we integrate \eqref{eq:wueq2}.
 We then have
 \[
 \begin{aligned}
  w^+_1 (u)^2
  = w^+_1 (u_0)^2 + 2 \int_{u_0}^u g_1(s) f(s) ds
  + 2 \int_{u_0}^u g_2(s) h(s,\mu) w^+_1 (s) ds.
 \end{aligned}
 \]
 It follows from \eqref{vwpmi0} that the right-hand side goes to $\infty$
 uniformly for $u \in [u_0, u_1^+)$ as $\mu \to \infty$.
 From this and \eqref{wzep}, we conclude that $u^+_1=u_2$ for large $\mu$ and
 $w^+_1(u;K,c,\mu) \to \infty$ uniformly for $u \in [u_0,u_2]$ as $\mu \to \infty$.
 Thus the assertion for $w^+_1$ is proved.
 The others can be shown in a similar way.
\end{proof}


We are now in a position to prove Theorem~\ref{thm:HECO}.


\begin{proof}[Proof of Theorem~\ref{thm:HECO}] 
 We fix $a \in (u_0,u_2)$ and consider the behavior of
 $\phi (\mu)\equiv (w^+_1 -w^+_2)|_{u=a}$.
 From Lemma~\ref{lem:mwmu}, we deduce that $u^+_1$ 
 is nondecreasing with respect to $\mu$.
 Since $u^+_1 > a$ for large $\mu$ from Lemma~\ref{lem:bwmi}, 
 there is some $b \in [-\infty, \infty)$ such that 
 $\{ \mu \in \r \ | \ u^+_1>a \} = (b, \infty)$.
 Using Lemmas~\ref{lem:mwmu} and \ref{lem:bwmi} again, we infer that
 $\phi_\mu > 0$, $\lim_{\mu \to \infty} \phi (\mu)=\infty$,
 and 
 \[
 \lim_{\mu \to b} \phi (\mu) =
 \left\{ 
 \begin{aligned}
  & - w^+_2(a;K,c,b) < 0
  & & \mbox{if } b \neq -\infty, \\
  & -\infty
  & & \mbox{if } b=-\infty
 \end{aligned}
 \right.
 \]
 because $u^+_1 = a$ and $w^+_1 (a) = 0$ for $\mu = b$ when $b \neq -\infty$.
 Therefore there exists a unique zero $\mu_b$ of $\phi$.
 We have thus proved the assertion for (\ref{eq1_TV}) with \eqref{eq:bcmi}.
 Since the same argument is valid for (\ref{eq1_TV}) with \eqref{eq:bcmd}, 
 we conclude Theorem~\ref{thm:HECO}.
\end{proof}


\begin{remark}
 By the implicit function theorem, 
 $\mu_b$ and $\mu_f$ are of class $C^1$ with respect to 
 $(K, c) \in \mathcal{D}_1$.
\end{remark}


\section{Behavior of $\mu_b$ and $\mu_f$}\label{sec_mubmuf}

This section is devoted to the proof of the following proposition.
We examine the behaviors of $\mu_b$ and $\mu_f$ 
as $c$ runs from $c_1$ to $c_M$ 
and prove the existence of a heteroclinic cycle in (\ref{eq1_TV}).


\begin{proposition}\label{prop:BMBF}
 Assume (\ref{eq1_cs}).
 For $K \in (K_0, K_M)$, there is a unique number
 $c_* = c_*(K) \in (c_1, c_M)$ such that
 \begin{equation}\label{eq4_muscs}
  \begin{aligned}
   & \mu_b < \mu_f \ \mbox{ if } \ c_1 < c < c_*, \\
   & \mu_b > \mu_f \ \mbox{ if } \ c_* < c < c_M, \\
   & \mu_b = \mu_f \ \mbox{ if } \ c = c_*.    
  \end{aligned}
 \end{equation}
\end{proposition}


We begin by showing the monotonicity of $w^\pm_j$ with respect to $c$.


\begin{lemma}\label{lem:mwc}
 For $(K,c) \in \mathcal{D}_1$ and $\mu \in \mathbb{R}$, 
 there hold $\mp (w^\pm_1)_c >0$ in $u \in (u_1, u^\pm_1)$ and 
 $\pm (w^\pm_2)_c >0$ in $u \in (u^\pm_2, u_2)$.
\end{lemma}


\begin{proof}
 We only estimate $W\equiv -(w^+_1)_c$.
 The other inequalities can be obtained in the same way.
 Differentiating the equalities $w^+_1(u_1)=0$ and $f(u_1)=0$ 
 with respect to $c$ gives $W (u_1) = \l_+ (u_1) / (K f'(u_1))$,
 where $\l_+ (u_1)$ was given in Lemma~\ref{lem:evus}.
 Since the right-hand side of this equality is positive,
 we infer that $W>0$ in a neighborhood of $u_1$. 
 Differentiating \eqref{eq:wueq} yields
 \[
 W_u =-\frac{g_1(u) f(u)}{(w^+_1)^2} W +\frac{g_1(u)}{Kw^+_1}. 
 \]
 Applying Lemma~\ref{lem:Wpes0},
 we conclude that $W>0$ in $(u_1,u^+_1)$.
 This completes the proof.
\end{proof}


Next we show the monotonicity of $\mu_b$ and $\mu_f$ with respect 
to $c \in (c_1,c_M)$.


\begin{lemma}\label{lem:mmuc}
 The inequalities $(\mu_b)_c>0$ and $(\mu_f)_c<0$ hold.
\end{lemma}


\begin{proof}
 We recall that the equality $w^+_1(u;K,c,\mu_b)=w^+_2(u;K,c,\mu_b)$ holds.
 Differentiating this with respect to $c$, we have
 \[
 (w^+_1)_c + (w^+_1)_\mu (\mu_b)_c = (w^+_2)_c + (w^+_2)_\mu (\mu_b)_c.
 \]
 This with Lemmas~\ref{lem:mwmu} and \ref{lem:mwc} yields
 \[
 (\mu_b)_c = \frac{(w^+_2)_c - (w^+_1)_c}{(w^+_1)_\mu - (w^+_2)_\mu} > 0.
 \]
 In a similar way, we have
 \[
 (\mu_f)_c = - \frac{(w^-_1)_c - (w^-_2)_c}{(w^-_1)_\mu -(w^-_2)_\mu} < 0.
 \]
 Therefore the lemma follows.
\end{proof}


Let us prove Proposition~\ref{prop:BMBF}.


\begin{proof}[Proof of Proposition~\ref{prop:BMBF}]
 For simplicity of notation, we ignore the dependence on $K$ and 
 write $\mu_b (c)$ instead of $\mu_b (K,c)$, for instance.
 By Lemma~\ref{lem:mmuc}, it suffices to show that
 \begin{equation}\label{ie:mbmf}
  \lim_{c \to c_M} \mu_b(c) >\lim_{c \to c_M} \mu_f(c),
   \quad \lim_{c \to c_1} \mu_b(c) <\lim_{c \to c_1} \mu_f(c).
 \end{equation}

 To derive the former inequality of \eqref{ie:mbmf}, we prove that
 \begin{equation}\label{es:mbfb}
  \lim_{c \to c_M} \mu_b (c) > \mu^*,
 \end{equation}
 where $\mu^*$ is determined by the relation
 \begin{equation}\label{def:msua}
  \int_{u_M}^{u_2^*} g_2 (u) h (u, \mu^*) du = 0
 \end{equation}
 for $u_2^* \equiv \lim_{c \to c_M} u_2(c) > 0$.
 To obtain a contradiction, suppose that \eqref{es:mbfb} is false.
 Then Lemma~\ref{lem:mmuc} yields
 $\mu_b (c) < \mu^*$ for all $c \in (c_1,c_M)$.
 Set $w_b (u) \equiv w^+_1 (u; c, \mu_b (c)) = w^+_2 (u; c, \mu_b (c))$.
 Since $w_b$ satisfies \eqref{eq:wueq2} for $\mu=\mu_b$,
 we can apply Lemma~\ref{lem:Wpes}
 with $W=w_b$, $A=2g_1 f - 2 (\mu^* -\mu_b) g_2 w_b$, 
 $B = 2 g_2 h (\cdot,\mu^*)$, $s_1=u_1$ and $s_2=u \in [u_1,u_2]$ to obtain
 \begin{equation}\label{es:wbub}
  \begin{aligned}
   w_b (u)^2
   & \le \int_{u_1}^u e^{u-s} (2 g_1 (s) f (s) + g_2 (s)^2 h (s, \mu^*)^2) ds
   - 2 (\mu^* -\mu_b) \int_{u_1}^u e^{u-s} g_2 (s) w_b (s) ds \\
   & \le \int_{u_1}^u e^{u-s} ( 2g_1 (s) f (s) + g_2 (s)^2 h (s,\mu^*)^2 ) ds.
  \end{aligned}
 \end{equation}
 If $u=u_0$, we find from \eqref{luacm} 
 that the right-hand side approaches $0$ as $c \to c_M$,
 which implies that $\lim_{c \to c_M} w_b (u_0 (c); c, \mu_b (c)) = 0$.
 Moreover, it follows from \eqref{es:wbub} that
 \begin{equation}\label{wbls}
  \limsup_{c \to c_M} \max_{u \in [u_1,u_2]} w_b (u) < \infty.
 \end{equation}
 We now integrate \eqref{eq:wueq} over $[u_0,u_2]$ to obtain
 \[
 w_b(u_0) +\int_{u_0}^{u_2} g_2(u)h(u,\mu_b) du
 = \int_{u_0}^{u_2} \frac{-g_1(u)f(u)}{w_b(u)}du. 
 \]
 By \eqref{def:msua} and the assumption $\lim_{c \to c_M}\mu_b(c) \le \mu^*$,
 we see that the left-hand side converges to some nonpositive number as $c \to c_M$.
 On the other hand, (\ref{wbls}) implies that the right-hand side is bounded below 
 by some positive constant for any $c$ close to $c_M$.
 This is a contradiction, and therefore \eqref{es:mbfb} holds.
 The inequality $\lim_{c \to c_M}\mu_f(c)<\mu^*$ can also be verified 
 in a similar way.
 We have thus shown the former inequality of \eqref{ie:mbmf}.

 Let us prove the latter inequality of \eqref{ie:mbmf}.
 In the case $K_1 < K < K_M$, we have $c_1 = c_m$ 
 and put $u_1^* \equiv \lim_{c \to c_m} u_1 (c) > 0$.
 Hence we can apply an argument similar to the proof of \eqref{es:mbfb} to obtain
 \[
 \lim_{c \to c_m} \mu_b(c) < \bar \mu^* < \lim_{c \to c_m} \mu_f(c),
 \]
 where $\bar \mu^*$ is determined by the relation
 \[
 \int_{u_1^*}^{u_m} g_2 (u) h (u, \bar \mu^*) du = 0.
 \]

 Next we assume (\ref{cd:nu}) and $K_0 < K \le K_1$.
 In this case, we have $c_1 = c_0$. 
 We prove that 
 \begin{equation}\label{mlc11}
  \lim_{c \to c_0} \mu_b (c) \le \mu_0 < \lim_{c \to c_0} \mu_f (c),
 \end{equation}
 where $\mu_0 \equiv - f' (0)$.
 Note that $\mu_0$ is given by $h (0, \mu_0) = 0$.
 Moreover, one can easily check that
 \[
 \lim_{c \to c_0} u_1 (c) = 0,
 \quad \lim_{c \to c_0} u_2 (c) > 0,
 \quad f'(u) > 0 \ \mbox{ for } \ u \in [0,u_M). 
 \]
 We prove the first inequality of (\ref{mlc11}) by contradiction.
 Obviously, there is a positive constant $\d_0$ such that 
 $h (u, \mu_b (c)) > 0$ for all $u \in [0, \d_0]$ and $c \in (c_0,c_M)$.
 By integrating \eqref{eq:wueq2} over $[u_1, \d_0]$,
 we deduce that
 \begin{equation}\label{mbles1}
  \frac{1}{2} w_b(\d_0)^2 \ge \int_{u_1}^{\d_0} g_1 (u) f (u) du.
 \end{equation}
 To estimate the left-hand side,
 we apply Lemma~\ref{lem:Wpes}
 with $W = w_b$, $A = 2 g_1 f - 2 (\mu_0 -\mu_b) g_2 w_b$, 
 $B = 2 g_2 h (\cdot, \mu_0)$, $s_1 = \d_0$ and $s_2 = u_2$.
 The result is
 \[
 \begin{aligned}
  w_b (\d_0)^2 
  \le \int_{\d_0}^{u_2} e^{s-\d_0}
  ( -2 g_1 (u) f (u) + g_2 (u)^2 h (u,\mu_0)^2 ) du
  + 2 (\mu_0 -\mu_b) \int_{\d_0}^{u_2} e^{s-\d_0} g_2 (u) w_b (u) du,
 \end{aligned}
 \]
 which implies that $w_b (\d_0)^2$ is bounded by some constant independent of $c$.
 On the other hand, the right-hand side of \eqref{mbles1} is estimated as
 \[
 \int_{u_1}^{\d_0} g_1(u)f(u) du
 \ge m_0 \int_{u_1}^{\d_0} (u-u_1) g_1(u) du
 \ge \frac{m_0}{2} \int_{2u_1}^{\d_0} u g_1(u) du, 
 \]
 where we have applied $f (u) \ge m_0 (u - u_1)$
 in all $u \in [u_1, \d_1] $ and $c \in [c_0, c_0 + \d_2]$
 for some positive constants $m_0, \d_1, \d_2$.
 By the assumption \eqref{cd:nu}, we conclude that the right-hand side 
 of \eqref{mbles1} diverges to $\infty$ as $c \to c_0$.
 This leads to a contradiction.

 A similar argument works for the second inequality of (\ref{mlc11}).
 Thus we obtain the latter inequality of \eqref{ie:mbmf}, and the proof is complete.
\end{proof}


\section{Structure of traveling pulse solutions}\label{sec:TPS}

Let us consider traveling pulse solutions of (\ref{eq1_TV}) with (\ref{eq:bctmp0}).
For simplicity of notation, we let 
\begin{equation}\label{def:tu1}
 \tilde u_1 = 
  \left\{
   \begin{aligned}
    & u_1 
    & & \mbox{if } \ (K,c) \in \mathcal{D}_1, \\
    & 0 
    & & \mbox{if } \ (K,c) \in \mathcal{D}_2.
   \end{aligned}
  \right.
\end{equation}
To prove Theorem~\ref{thm:HOCO},
we examine the properties of $u^\pm_j$.
We note that the derivatives $(u^\pm_1)_\mu$ (resp. $(u^\pm_2)_\mu$) 
exist if $u^\pm_1 \in (u_0, u_2)$ (resp. $u^\pm_2 \in (\tilde u_1,u_0)$),
thanks to the smooth dependence of stable and unstable manifolds 
of the equilibria $(u_1, 0)$ and $(u_2, 0)$ on parameters.


\begin{lemma}\label{lem:mwmc}
 It holds that $\pm (u^\pm_1)_\mu > 0$ if $u^\pm_1 \in (u_0,u_2)$
 and $\pm (u^\pm_2)_\mu > 0$ if $u^\pm_2 \in (\tilde u_1,u_0)$. 
\end{lemma}


\begin{proof}
 We only prove the assertion for $u^+_1$.
 The others can be handled in a similar way.
 For abbreviation, we write $w$ instead of $w_1^+$.

 Assume $u^+_1 \in (u_0,u_2)$ and set $\zeta = \zeta (u) \equiv (w^2)_\mu$.
 We first show that $\zeta$ is continuously extended up to the point $u = u^+_1$
 and $\zeta (u^+_1) > 0$.
 It is easy to see from Lemma~\ref{lem:mwmu}
 that $\zeta$ is positive if $u$ is bigger than and close to $u_1$.
 By differentiating \eqref{eq:wueq2} with respect to $\mu$,
 we see that $\zeta$ satisfies
 \[
 \zeta_u = \frac{g_2 (u) h (u,\mu)}{w} \zeta + 2 g_2 (u) w.
 \]
 Let us check that $g_2 h (\cdot,\mu) / w$ is locally integrable on $(u_1, u^+_1]$.
 From \eqref{eq:wueq2} and \eqref{wzep}, we have
 $\lim_{u \to u^+_1} w w_u = g_1 (u^+_1) f (u^+_1)$.
 Since $f (u^+_1) < 0$, 
 there is a constant $C > 0$ such that $w (u) \ge C \sqrt{u^+_1 - u}$
 for $u$ close to $u^+_1$, which implies the local integrability of 
 $g_2 h (\cdot,\mu) / w$.
 Then it follows from Lemma~\ref{lem:Wpes0} that 
 $\zeta$ is continuously extended up to $u^+_1$ and positive in $(u_1, u^+_1]$.

 Recall that $u^+_1$ is determined implicitly by $w (u^+_1)^2 = 0$.
 Differentiating this equality with respect $\mu$ and then using \eqref{eq:wueq2},
 we see that
 \[
 2 g_1 (u^+_1) f (u^+_1) (u^+_1)_\mu + \zeta (u^+_1) = 0.
 \]
 Therefore we obtain
 \[
 (u^+_1)_\mu = - \frac{\zeta (u^+_1)}{2 g_1(u^+_1) f (u^+_1)} > 0,
 \]
 which completes the proof.
\end{proof}


\begin{lemma}\label{lem:umom}
 Define
 \[
 \mu^+_1 \equiv \inf \{ \mu \in \r \ | \ u^+_1 > u_0 \},
 \quad \mu^-_1 \equiv \sup \{ \mu \in \r \ | \ u^-_1 > u_0 \}
 \]
 for $(K,c) \in \mathcal{D}_1$ and 
 \[
 \mu^+_2 \equiv \sup \{ \mu \in \r \ | \ u^+_2 < u_0 \},
 \quad \mu^-_2 \equiv \inf \{ \mu \in \r \ | \ u^-_2 < u_0 \}
 \]
 for $(K,c) \in \mathcal{D}_1 \cup \mathcal{D}_2$.
 Then $\mu^+_1, \mu^-_2 \in [-\infty, \infty)$
 and $\mu^-_1, \mu^+_2 \in (-\infty, \infty]$.
 Moreover, one has
 \begin{equation}\label{umom}
  \mu^+_1<\mu^-_1, \quad \mu^+_2>\mu^-_2.
 \end{equation}
\end{lemma}


\begin{proof}
 Assume $(K,c) \in \mathcal{D}_1$.
 The first statement follows immediately from Lemma~\ref{lem:bwmi}.
 We prove \eqref{umom}.
 On the contrary, suppose that $\mu^+_1 \ge \mu^-_1$.
 From Lemma~\ref{lem:mwmc}, we see that $u^+_1$ and $u^-_1$ 
 are nondecreasing and nonincreasing with respect to $\mu$, respectively.
 Hence $u^+_1=u^-_1=u_0$ for all $\mu \in [\mu^-_1,\mu^+_1]$.
 By \eqref{eq:wueq}, we have
 \[
 (w^+_1 -w^-_1)_u =g_1(u)f(u)\left( \frac{1}{w^+_1} -\frac{1}{w^-_1}\right).
 \]
 Since the right-hand side is positive in $(u_1,u_0)$, we deduce that
 \[
 w^+_1 (u_0) - w^-_1 (u_0) > w^+_1 (u_1) - w^-_1 (u_1).
 \]
 This leads to a contradiction because 
 $w^+_1 (u_0) = w^-_1 (u_0) = 0$ and $w^+_1 (u_1) = w^-_1 (u_1) = 0$.
 Therefore the former inequality of \eqref{umom} holds.
 The latter inequality can be shown in a similar way.

 Next we assume $(K,c) \in \mathcal{D}_2$.
 By the same argument as in the proof of Lemma~\ref{lem:mwmc}, it follows that 
 $u_2^- < a$ (resp. $u_2^+ < a$) if $\mu$ (resp. $-\mu$) is large enough.
 Then we easily verify (\ref{umom}) in the same manner as above.
\end{proof}


We give sufficient conditions on $\mu$ for which $u_2^\pm$ is positive
in the case of $(K,c) \in \mathcal{D}_2$.
We use the notation $\mu_0 = - f' (0)$, which has already been defined in 
the proof of Proposition~\ref{prop:BMBF}.


\begin{lemma}\label{u2pmp}
 Let $(K,c) \in \mathcal{D}_2$ and assume \eqref{cd:nu}.
 Then $u^+_2 > 0$ if $\mu \ge \mu_0$, while $u^-_2 > 0$ if $\mu \le \mu_0$. 
\end{lemma}


\begin{proof}
 Let $\mu \ge \mu_0$. 
 Then there is $\d_0 > 0$ such that $h (u, \mu) \ge 0$ for $u \in [0,\d_0]$.
 To obtain a contradiction, we assume that $w_2^+ (u)$ is positive in $(0, u_2)$.
 For any $u \in (0, \d_0)$, 
 we integrate \eqref{eq:wueq2} over $[u, \d_0]$ and then have
 \begin{equation}\label{es:u2pmp}
  w^+_2 (\d_0)^2 - w^+_2 (u)^2 \ge 2 \int_{u}^{\d_0} g_1 (s) f (s) ds.
 \end{equation} 
 From (ii) of Lemma~\ref{lem:prf}, we can choose a constant $C>0$ 
 such that $f (u) \ge C u$ for all $u$ close to $0$. 
 Then it follows that the right-hand side of \eqref{es:u2pmp}
 diverges to $\infty$ as $u \to 0$ by the assumption \eqref{cd:nu}.
 On the other hand, the left-hand side of (\ref{es:u2pmp}) is bounded above 
 as $u \to 0$, which is a contradiction.
 Therefore $w^+_2$ must vanish at some point in $(0, u_2)$,
 which implies $u^+_2 > 0$.

 In a similar way, we can also show that $u^-_2>0$ if $\mu \le \mu_0$.
 Therefore the lemma follows.
\end{proof}


Before proceeding to the proof of Theorem~\ref{thm:HOCO}, 
we first give necessary conditions for the existence of solutions 
to the problem (\ref{eq1_TV}) with \eqref{eq:bctmp0}.
We will apply the following arguments not only for homoclinic orbits
but also for periodic orbits discussed in the next section.


\begin{lemma}\label{lem:neho}
 If $(K,c) \not \in \mathcal{D}_1 \cup \mathcal{D}_2$,
 then there is no solution $(u, w)$ of (\ref{eq1_TV}) 
 satisfying (\ref{eq:bctmp0}) and $u (z) > 0$ in $-\infty < z < \infty$.
\end{lemma}


\begin{proof}
 From the assumption, $(K, c) \in \mathcal{D}_3 \cup \mathcal{D}_4$,
 where 
 \[
 \begin{aligned}
  \mathcal{D}_3 & \equiv \{ (K, c) \ | \ K < 0, \ c \ge c_0 \}
  \cup \{ (K, c) \ | \ 0 < K \le K_1, \ c \le c_m (K) \}
  \cup \{ (K, c) \ | \ K_1 < K, \ c \le c_0 \}, \\
  \mathcal{D}_4 & \equiv \{ (K, c) \ | \ K<0, \ c < c_0 \}
  \cup \{ (K, c) \ | \ 0 < K \le K_0, \ c \ge c_0 \}
  \cup \{ (K, c) \ | \ K_0 < K <K_M, \ c \ge c_M(K) \} \\
  & \quad \cup \{ (K, c) \ | \ K_1 < K < K_M, \ c_0 < c \le c_m (K) \}
  \cup \{ (K, c) \ | \ K \ge K_M, \ c > c_0 \}.  
 \end{aligned}
 \]
 We easily see that $f (u) \ge 0$ in $u > 0$ 
 if $(K, c) \in \mathcal{D}_3$.
 On the other hand, if $(K, c) \in \mathcal{D}_4$, 
 then there is $u_* > 0$ such that $f (u) \le 0$ in $0 \le u < u_*$,
 while $f (u) \ge 0$ in $u \ge u_*$.

 To obtain a contradiction, 
 suppose that (\ref{eq1_TV}) has a solution $(u,w)$ satisfying (\ref{eq:bctmp0}).
 We first consider the case $(K, c) \in \mathcal{D}_3$.
 Let $H$ be a primitive of $g_2 h (\cdot ,\mu)$.
 Then we have
 \[
 (w-H(u))_z =w_z -g_2(u) h(u,\mu) u_z =g_1(u) f(u).
 \]
 Integrating this over $(-\infty, \infty)$ and using (\ref{eq:bctmp0}),
 we deduce that
 \begin{equation}\label{eqne}
  0 = \int_{-\infty}^{\infty} g_1 (u(z)) f (u(z)) d z.
 \end{equation}
 Then we obtain $f (u(z)) \equiv 0$, which contradicts the condition
 $(u(z), w(z)) \not \equiv (\overline{u}, 0)$.

 Next we assume $(K, c) \in \mathcal{D}_4$.
 From (\ref{eq:bctmp0}),
 we see that $u$ has either a global maximum or a global minimum.
 Suppose that $u$ has a global maximum at some $z_0$.
 Then we have 
 \[
 w (z_0) = u_z (z_0) = 0, \quad w_z (z_0) = u_{zz}(z_0) \le 0. 
 \]
 Substituting these into the second equality of (\ref{eq1_TV}) 
 yields $g_1 (u (z_0)) f (u (z_0)) \le 0$.
 Since $(u, w)$ is not an equilibrium, we see $f (u (z_0)) < 0$
 and then $u (z) \le u (z_0) \le u_*$.
 However, it follows from the equality \eqref{eqne} that $f (u(z)) \equiv 0$,
 which is a contradiction.
 The other case can be derived in the same way as above.
 Thus the proof is complete.
\end{proof}


We are now in a position to show Theorem~\ref{thm:HOCO}.


\begin{proof}[Proof of Theorem~\ref{thm:HOCO}]
 The assertion (i) is a direct consequence of Lemma~\ref{lem:neho}.
 We begin with the proof of (ii).
 On the contrary, suppose that there exists a solution $(u, w)$ 
 of (\ref{eq1_TV}) satisfying \eqref{eq:bctmp0} with $\overline{u} = u_0$.
 Then we must have $\mu = - f' (u_0)$
 because (ii) of Lemma~\ref{lem:evus} shows that any solution of (\ref{eq1_TV})
 cannot converge to $(u_0, 0)$ as either $z \to \infty$ or $z \to -\infty$
 if $\mu \neq - f' (u_0)$.
 Let $z_* > 0$ be sufficiently large. 
 Clearly, $(u, w)$ is close to $(u_0, 0)$ in $|z| \ge z_*$.
 More precisely, there are $z_0$ and $r = r (z) > 0$ such that 
 $r \to 0$ as $|z| \to \infty$ and $(u, w)$ is approximated by 
 \[
 (u_0, 0) + r (\cos \omega_0 (z - z_0), - \omega_0 \sin \omega_0 (z - z_0))
 \]
 in $|z| \ge z_*$ from (ii) of Lemma~\ref{lem:evus}.
 Then the orbit of $(u, w)$ must intersect with itself,
 which leads to the contradiction because of 
 the uniqueness of a solution in ordinary differential equations.

 Let us turn to the proofs of (iii) and (iv).
 First we consider the case $(K,c) \in \mathcal{D}_1$ and $\overline{u}=u_1$.
 It is sufficient to check the condition
 \begin{equation}\label{crho}
  u^+_1 = u^-_1 \in (u_0, u_2).
 \end{equation}
 Recall that
 \[
 \begin{aligned}
  & u^+_1 = u_2 \ \mbox{ and } \ w^+_1(u_2)=0 
  \ \mbox{ if and only if } \ \mu=\mu_b,\\
  & u^-_1=u_2 \ \mbox{ and } \ w^-_1(u_2)=0 
  \ \mbox{ if and only if } \ \mu=\mu_f.   
 \end{aligned}
 \]
 These with Lemma~\ref{lem:mwmc} imply that
 \begin{equation}\label{upmr}
  \begin{aligned}
   & u^+_1 
   \left\{
   \begin{aligned}
    & \in (u_0,u_2)
    & & \mbox{if } \ \mu^+_1 < \mu < \mu_b, \\
    & = u_2
    & & \mbox{if } \ \mu \ge \mu_b,
   \end{aligned}
   \right.
   \\
   & u^-_1 
   \left\{
   \begin{aligned}
    & = u_2 
    & & \mbox{if } \ \mu \le \mu_f, \\
    & \in (u_0,u_2)
    & & \mbox{if } \ \mu_f < \mu < \mu^-_1.
   \end{aligned}
   \right.
  \end{aligned}
 \end{equation}
 We also recall that
 \begin{equation}\label{rbmbf}
  \begin{aligned}
   & \mu_b \le \mu_f \ \mbox{ if } 
   \ (K,c) \in \mathcal{D}_{1,2} \cup \mathcal{D}_{1,3}, \\
   & \mu_b>\mu_f \ \mbox{ if } 
   \ (K,c) \in \mathcal{D}_{1,1}.   
  \end{aligned}
 \end{equation}
 Combining \eqref{umom}, \eqref{upmr}, and \eqref{rbmbf}, 
 we conclude that (\ref{crho}) is never satisfied 
 if $(K,c) \in \mathcal{D}_{1,2} \cup \mathcal{D}_{1,3}$,
 while (\ref{crho}) holds for some $\mu \in (\mu_f, \mu_b)$ 
 if $(K,c) \in \mathcal{D}_{1,1}$.
 The uniqueness of $\mu$ satisfying (\ref{crho}) follows from Lemma~\ref{lem:mwmc}.
 Therefore the assertion is proved in this case.

 Next, we examine the case $\overline{u}=u_2$.
 By the same argument applied above,
 we can show the unique existence of $\mu^2_{pul} \in (\mu_b,\mu_f)$ 
 for $(K,c) \in \mathcal{D}_{1,2}$
 and the nonexistence of solutions 
 for $(K,c) \in \mathcal{D}_{1,1} \cup \mathcal{D}_{1,3}$.
 Hence we only need to consider the case of 
 $(K,c) \in \mathcal{D}_2$ under (\ref{cd:nu}).
 Put
 \[
 \underline{\mu} \equiv \inf \{ \mu \in \r \ | \ u^+_2 > 0 \},
 \quad \overline{\mu} \equiv \sup \{ \mu \in \r \ | \ u^-_2 > 0 \}.
 \]
 From Lemmas~\ref{lem:umom} and \ref{u2pmp}, 
 we see $\underline{\mu} < \mu^+_2$ and $\mu^-_2 < \overline{\mu}$.
 We then have
 \begin{equation}\label{upmr2}
  u^+_2 \in (0,u_0) \ \mbox{if } \underline{\mu} < \mu < \mu^+_2,
   \quad u^-_2 \in (0,u_0) \ \mbox{if } \mu^-_2 < \mu < \overline{\mu}
 \end{equation}
 by Lemma~\ref{lem:mwmc}.
 Using Lemma~\ref{u2pmp} and the fact that $\{ \mu \in \r \ | \ u^-_2>0\}$ is open, 
 we see that
 \begin{equation}\label{oumu}
  \underline{\mu} < \mu_0 < \overline{\mu}.
 \end{equation}
 Combining \eqref{umom}, \eqref{upmr2} and \eqref{oumu}, we obtain 
 $\mu^2_{pul}$ such that $u^+_2=u^-_2 \in (0,u_0)$ for $\mu=\mu^2_{pul}$.
 Thus the proof is complete.
\end{proof}


We conclude this section by deriving an estimate of $\mu_{pul}^j$
to be used in Section~\ref{sec_bifurcation}.


\begin{lemma}\label{lem:empl}
 For $\mu = \mu_{pul}^j$ ($j = 1, 2$), let $(u, w)$ be the traveling pulse 
 solution obtained in (iv) of Theorem~\ref{thm:HOCO}.
 Then
 \[
 - \sup_{u \in (\underline{m},\overline{m})} f'(u)   
 < \mu_{pul}^j < - \inf_{u \in (\underline{m},\overline{m})} f' (u), 
 \]
 where 
 \[
 \underline{m}\equiv \inf_{z \in \r} u(z),
 \quad \overline{m}\equiv \sup_{z \in \r} u(z). 
 \]
\end{lemma}


\begin{proof}
 Let $D \subset \r^2$ be the region enclosed by the closed curve 
 $\{(u(z), w(z))\}_{z \in \r} \cup \{(u_j, 0)\}$ for $j = 1, 2$,
 and let $F (u,w)$ be the two-dimensional vector field defined in \eqref{def:FJ}.
 It is well-known that 
 \[
 \int_D \nabla \cdot F (u,w) d u d w = 0,
 \]
 which follows easily from the divergence theorem and (\ref{eq1_TV}).
 Since $\nabla \cdot F (u,w)=g_2 (u) h (u, \mu_{pul}^j)$,
 $h (\cdot, \mu_{pul}^j)$ must change its sign in $(\underline{m}, \overline{m})$.
 Therefore we obtain
 \[
 \inf_{u \in (\underline{m},\overline{m})} h(u, \mu_{pul}^j) < 0 
 < \sup_{u \in (\underline{m},\overline{m})} h(u, \mu_{pul}^j), 
 \]
 which concludes the lemma.
\end{proof}


\section{Structure of periodic traveling wave solutions}
\label{sec_Structure of periodic traveling wave solutions}

In this section, we study periodic traveling wave solutions.
The proof of Theorem~\ref{thm:ptws} is similar to that of Theorem~\ref{thm:HOCO}.
Let $(u, w)$ be a solution of (\ref{eq1_TV}) with 
the initial condition $(u(0), w(0)) = (q, 0)$.
First, we assume that $(K,c) \in \mathcal{D}_1$ and $q \in (u_1, u_0)$.
It is then seen from Lemma~\ref{lem:prf} that 
$(u (z), w (z)) \in S_+$ and $(u (-z),w (-z)) \in S_-$
for small $z > 0$, where $S_- \equiv  \{ (u, w) \ | \ q < u < u_2, \ w<0\}$ 
and $S_+ \equiv \{ (u, w) \ | \ q < u < u_2, \ w > 0 \}$.
Hence we define $z^{\pm}$ by 
\begin{equation}\label{eq6_zsud}
 \begin{aligned}
  & z^+ \equiv \sup \{ z_0 \ | \ \{(u (z), w (z)) \}_{0 < z < z_0} \subset S_+ \}
  \in (0,\infty], \\
  & z^- \equiv \inf \{ z_0 \ | \ \{(u (z), w (z)) \}_{z_0 < z < 0} \subset S_- \}
  \in [-\infty, 0).
 \end{aligned}
\end{equation}
Lemma~\ref{lem:prf} leads to 
\[
\begin{aligned}
 & (u (z^+), w (z^+)) 
 \in \{ (u, 0) \ | \ u_0 \le u \le u_2 \} \cup \{ (u_2, w) \ | \ w > 0 \}, \\
 & (u (z^-), w (z^-)) 
 \in \{ (u, 0) \ | \ u_0 \le u \le u_2 \} \cup \{ (u_2, w) \ | \ w < 0 \}.
\end{aligned}
\]
Since $u_z \neq 0$ in $z \in (0, z^+) \cup (z^-, 0)$, 
there are $u^{\pm} \in [u_0, u_2]$ and functions $w^{\pm} = w^{\pm} (u)$ such that 
\[
\begin{aligned}
 \{ (u (z), w (z)) \ | \ 0 < z < z^+ \}
 & = \{ (u, w^+ (u)) \ | \ q < u < u^+ \}, \\ 
 \{ (u (z), w (z)) \ | \ z^- < z < 0 \}
 & = \{ (u, w^- (u)) \ | \ q < u < u^- \}. 
\end{aligned}
\]
To emphasize the dependency of the parameters, we may write
$u^{\pm} = u^{\pm} (K, c, \mu)$ and $w^{\pm} = w^{\pm} (u; K, c, \mu)$.
By definition, we see that $w^{\pm}$ satisfy (\ref{eq:wueq}), \eqref{eq:wueq2} 
and $\displaystyle \lim_{u \to q} w^{\pm} (u) = 0$.
Moreover, $w^\pm$ are continuously differentiable with respect to $(K, c, \mu)$.

One can similarly define $z^\pm$, $u^{\pm}$ and $w^{\pm}$
for $(K,c) \in \mathcal{D}_1 \cup \mathcal{D}_2$ and $q \in (u_0, u_2)$.
In this case, we have $\pm z^\pm < 0$, $\pm w^{\pm} > 0$, 
$u^{\pm} \in [\tilde u_1, u_0]$, and
\[
\begin{aligned}
 & \{ (u (z), w (z)) \ | \ z^+ < z < 0 \}
 = \{ (u, w^+ (u)) \ | \ u^+ < u < q \}, \\
 & \{ (u (z), w (z)) \ | \ 0 < z < z^- \}
 = \{ (u, w^- (u)) \ | \ u^- < u < q \},
\end{aligned}
\]
where $\tilde u_1$ was given in (\ref{def:tu1}).


\begin{proof}[Proof of Theorem~\ref{thm:ptws}]
 In the same way as in the proof of Lemma~\ref{lem:mwmc}, we easily verify 
 \begin{equation}\label{eq6_lem:prwq}
  \pm (u^{\pm})_\mu > 0.
 \end{equation}
 Let $(K,c) \in \mathcal{D}_{1,1} \cup \mathcal{D}_{1,3}$ and $q \in (u_1,u_0)$.
 It is clear that $(u,w)$ is a periodic solution of (\ref{eq1_TV}) 
 with \eqref{cd:rq} if and only if $u^{+} = u^{-} \in (u_0,u_2)$.
 By the same argument as in the proof of Lemma~\ref{lem:bwmi},
 if $\mu$ (resp. $-\mu$) is large enough for arbitrarily fixed $a \in (q, u_2)$,
 $u^+ > a$ (resp. $u^- > a$).
 Define
 \[
 \begin{aligned}
  & \underline{\mu}^{+} \equiv \inf \{ \mu \in \r \ | \ u^+ > u_0 \},
  & & \overline{\mu}^{-} \equiv \sup \{ \mu \in \r \ | \ u^- > u_0 \}, \\
  & \overline{\mu}^{+} \equiv \inf \{ \mu \in \r \ | \ u^+ = u_2 \},
  & & \underline{\mu}^{-} \equiv \sup \{ \mu \in \r \ | \ u^- = u_2 \}.
 \end{aligned}
 \]
 Then $-\infty \le \underline{\mu}^{+} < \overline{\mu}^{+} < \infty$
 and $-\infty < \underline{\mu}^{-} < \overline{\mu}^{-} \le \infty$.
 By (\ref{eq6_lem:prwq}), we have
 \[
 u^{+} 
 \left\{
 \begin{aligned}
  & \in (u_0,u_2)
  & & \mbox{if} \ \underline{\mu}^{+} < \mu < \overline{\mu}^{+}, \\
  & = u_2
  & & \mbox{if} \ \mu \ge \overline{\mu}^{+},
 \end{aligned}
 \right.
 \]
 \[
 u^{-} 
 \left\{
 \begin{aligned}
  & = u_2 
  & & \mbox{if} \ \mu \le \underline{\mu}^{-}, \\
  & \in (u_0,u_2)
  & & \mbox{if} \ \underline{\mu}^{-} < \mu < \overline{\mu}^{-}.
 \end{aligned}
 \right. 
 \]
 It is therefore sufficient to show that
 \begin{equation}\label{ruqpm}
  \underline{\mu}^{+} < \overline{\mu}^{-},
   \quad \overline{\mu}^{+} > \underline{\mu}^{-}.
 \end{equation}
 The former inequality above
 is shown in the same way as the proof of Lemma~\ref{lem:umom}.
 To derive the latter inequality, we observe that $u^{+} < u^+_1$
 when $u^{+} \in (u_0, u_2)$, which follows from the fact that the orbits 
 $\{ (u, w^+_1(u)) \ | \ u_1 < u < u^+_1 \}$ and 
 $\{ (u, w^{+} (u)) \ | \ q < u < u^{+} \}$ cannot intersect.
 In particular, we have $u^{+} < u^+_1 = u_2$ if $\mu = \mu_b$.
 Hence it follows that $\overline{\mu}^{+} > \mu_b$.
 In a similar manner, we also have $\underline{\mu}^{-}<\mu_f$.
 Since the inequality $\mu_b \ge \mu_f$ holds 
 under the condition $(K,c) \in \mathcal{D}_{1,1} \cup \mathcal{D}_{1,3}$,
 we obtain the latter inequality of \eqref{ruqpm}.
 Therefore we conclude that for each $q$ satisfying (\ref{cd:rq}), 
 there is $\mu_{per}$ such that a periodic solution of (\ref{eq1_TV})
 with $(u(0), w(0)) = (q, 0)$ exists for $\mu = \mu_{per}$.
 The uniqueness of $\mu_{per}$ follows from (\ref{eq6_lem:prwq}).

 In the same way as in the proof of Lemma~\ref{lem:neho},
 we readily see that the condition $(K, c) \in \mathcal{D}_1 \cup \mathcal{D}_2$ 
 is a necessary condition for the existence of a periodic solution 
 of (\ref{eq1_TV}) with (\ref{cd:rq}).
 Finally we prove that there exists no periodic solutions 
 if $\mu \neq \mu_{per}$ in (\ref{eq1_TV}) with (\ref{cd:rq}).
 We first assume that $(K,c) \in \mathcal{D}_{1,1} \cup \mathcal{D}_{1,3}$.
 Let $(u, w)$ be a periodic solution of (\ref{eq1_TV}) with the period $Z > 0$.
 Then there exist $q$ and $z_0$ such that $(u (z_0), w (z_0)) = (q, 0)$.
 If $q$ is not in $[u_1, u_2]$, we see from Lemma~\ref{lem:prf} that
 \[
 \{ (u (z),w (z)) \ | \ z > z_0 \} \subset 
 \left\{
 \begin{aligned}
  & \{ (u, w) \ | \ 0 < u < q, w < 0 \} & & \mbox{ if } \ q < u_1, \\
  & \{ (u, w) \ | \ u > q, w > 0\} & & \mbox{ if } \ q > u_2,
 \end{aligned}
 \right. 
 \]
 contrary to the fact that $(u (z_0 + Z), w (z_0 + Z)) = (q, 0)$.
 Hence we have $q \in (u_1, u_0) \cup (u_0, u_2)$ 
 because $(u, w)$ is not an equilibrium.
 If $q \in (u_1,u_0)$, then we must have $\mu = \mu_per$ by the above argument.
 Hence a periodic solution exists only for $\mu = \mu_{per}$.
 If $q \in (u_0, u_2)$, the orbit $\{ (u (z),w (z)) \ | \ z > z_0 \}$ 
 meets a point $(\tilde q, 0)$ with $\tilde q \in (u_1, u_0)$,
 since otherwise, one could show by Lemma~\ref{lem:prf} that
 \[
 \{ (u (z),w (z)) \ | \ z > z_0 \}
 \subset \{ (u, w) \ | \ u < q, w < 0 \}, 
 \]
 contrary to the fact that $(u (z_0 + Z), w (z_0 + Z)) = (q, 0)$.
 Therefore we again have $\mu = \mu_{per}$.

 We omit the discussion for the other cases
 because the same argument as above can be applied.
 Thus the proof is complete.
\end{proof}


\section{Bifurcations of traveling wave solutions}
\label{sec_bifurcation}

We have discussed several types of traveling wave solutions 
in Sections~\ref{sec:TF}--\ref{sec_Structure of periodic traveling wave solutions}.
It is then natural to investigate connections between them.
In this section, we observe that some of the solutions converge to other solutions
when parameters approach specific values.
This study provides information on the structure of solutions 
in a bifurcation diagram.

To state the results of this section, we introduce some notation.
Let $(u^j_{pul},w^j_{pul})$ be the homoclinic orbit of (\ref{eq1_TV}) 
with \eqref{eq:bctmp0} for $\mu = \mu^j_{pul}$ and $\overline{u} = u_j$
in $j=1, 2$, which is obtained in Theorem~\ref{thm:HOCO}.
Similarly, $(u_{per},w_{per})$ denotes
the solution of (\ref{eq1_TV}) with $(u_{per} (0), w_{per} (0)) = (q, 0)$ 
for $q$ satisfying \eqref{cd:rq} and $\mu = \mu_{per}$
as seen in Theorem~\ref{thm:ptws}.
We define $Z_{per}$ to be the (fundamental) period of $(u_{per},w_{per})$.
Moreover, set
\[
\mathcal{O}^j_{pul} = \{ (u^j_{pul} (z), w^j_{pul} (z)) \ | \ z \in \mathbb{R} \},
\]
\[
\mathcal{O}_{per} = \{ (u_{per} (z), w_{per} (z)) \ | \ 0 \le z < Z_{per} \},
\]
which represent the homoclinic 
and the periodic orbits in the phase plane, respectively.
To emphasize the dependency of the parameters, 
we may write $\mathcal{O}^j_{pul} = \mathcal{O}^j_{pul} (K,c)$ 
and $\mathcal{O}_{per} = \mathcal{O}_{per} (K,c,q)$.
Let $\mathcal{O}_*$ be the heteroclinic cycle in (\ref{eq1_TV}) 
for $(c, \mu) = (c_*, \mu_*)$ consisting of two heteroclinic orbits 
connecting the equilibrium points $(u_1,0)$, $(u_2,0)$.
Note that $\mathcal{O}_* = \mathcal{O}^{+}_1 \cup \mathcal{O}^{-}_1$ 
($= \mathcal{O}^{+}_2 \cup \mathcal{O}^{-}_2$),
where $\mathcal{O}^{\pm}_j$ for $j = 1, 2$ were defined in Section~\ref{sec:TF}.

We also introduce the notion of convergence for sets in $\mathbb{R}^2$.
Let $I \subset \r$ be an interval and let $\tau_0 \in \overline{I}$.
For $A \subset \mathbb{R}^2$ and 
$\{A_\tau\}_{\tau \in I} \subset \mathbb{R}^2$,
the notation $A_\tau \to A$ as $\tau \to \tau_0$ 
is used if $\{ A_\tau \}$ converges to $A$ 
with respect to the Hausdorff distance in $\mathbb{R}^2$ (\cite{MR1335452}), that is, 
\[
\max \left\{ \sup_{a \in A} \inf_{b \in A_\tau} |a-b|, 
\sup_{b \in A_\tau} \inf_{a \in A} |a-b| \right\} \to 0
\ \mbox{ as } \ \tau \to \tau_0.
\]
We note that if $A$ consists of a single point $(\overline{u},\overline{w})$ 
and $A_\tau$ is an orbit $\{ (u^\tau(z), w^\tau(z)) \ | \ z \in \r \}$,
then the convergence of $A_\tau$ to $A$ means that
$(u^\tau(z),w^\tau(z)) \to (\overline{u},\overline{w})$
uniformly for $z \in \mathbb{R}$ as $\tau \to \tau_0$.

The goal of this section is to present two propositions. 
First, we examine the relationship between the homoclinic orbits 
and the heteroclinic cycle (Proposition~\ref{prop:hobp}). 
Next, we show that the periodic orbit converges to the homoclinic orbit
when $q$ approaches the equilibrium point (Proposition~\ref{prop:pobp}).


\begin{proposition}\label{prop:hobp}
 Assume (\ref{eq1_cs}).
 Then, for $j = 1,2$,
 \begin{equation}\label{hobp0}
  \mu^j_{pul} (K,c) \to \mu_*,
   \quad
   \mathcal{O}^j_{pul} (K,c) \to \mathcal{O}_*
   \ \mbox{ as } \ c \to c_*.
 \end{equation}
 Furthermore, there hold
 \begin{equation}\label{hobp1}
  \mu^1_{pul} (K,c) \to 0, 
  \quad \mathcal{O}^1_{pul} (K,c) \to \{ (u_M,0) \} \ \mbox{ as } \ c \to c_M,
 \end{equation}
 \begin{equation}\label{hobp2}
  \mu^2_{pul} (K,c) \to 0,
  \quad \mathcal{O}^2_{pul} (K,c) \to \{ (u_m,0) \} \ \mbox{ as } \ c \to c_m.  
 \end{equation}
\end{proposition}


\begin{proposition}\label{prop:pobp}
 Assume the condition (\ref{eq1_cs}).
 Then there holds
 \begin{equation}\label{pobp1}
  \mu_{per} \to - f' (u_0),
   \ Z_{per} \to \frac{2 \pi}{\omega_0},
   \ \mathcal{O}_{per} \to \{ (u_0, 0) \}
 \end{equation}
 as $q \to u_0$, where $\omega_0$ was given in Lemma~\ref{lem:evus}.
 Moreover,
 \begin{equation}\label{pobp21}
  \mu_{per} \to \mu^1_{pul},
   \ Z_{per} \to \infty, 
   \ \mathcal{O}_{per} \to \mathcal{O}^1_{pul}
   \ \mbox{ as } \ q \to u_1
 \end{equation}
 if $(K,c) \in \mathcal{D}_{1,1}$,
 \begin{equation}\label{pobp22}
  \mu_{per} \to \mu^2_{pul},
   \ Z_{per} \to \infty, 
   \ \mathcal{O}_{per} \to \mathcal{O}^2_{pul}
   \ \mbox{ as } \ q \to u_2
 \end{equation}
 if $(K,c) \in \mathcal{D}_{1,2} \cup \mathcal{D}_2$,
 and 
 \begin{equation}\label{pobp23}
  \mu_{per} \to \mu_*,
   \ Z_{per} \to \infty, 
   \ \mathcal{O}_{per} \to \mathcal{O}_*
   \ \mbox{ as } \ q \to u_1, u_2
 \end{equation}
 if $(K,c) \in \mathcal{D}_{1,3}$.
\end{proposition}


\begin{remark}
 \begin{itemize}
  \item[(i)]
	    Bifurcations of homoclinic and heteroclinic orbits are observed:
	    \eqref{hobp0} indicates that the homoclinic orbits 
	    $\mathcal{O}^1_{pul}$ and $\mathcal{O}^2_{pul}$ 
	    bifurcate from the heteroclinic cycle $\mathcal{O}_*$;
	    \eqref{pobp21} indicates that the periodic orbit $\mathcal{O}_{per}$ 
	    becomes the homoclinic orbit $\mathcal{O}^j_{pul}$ 
	    or the heteroclinic cycle $\mathcal{O}_*$
	    when the initial value $q$ approaches $u_j$.
	    A Hopf bifurcation is also observed:
	    \eqref{pobp1} shows that $\mathcal{O}_{per}$ bifurcates from $(u_0,0)$.
  \item[(ii)]
	     We note that the equalities
	     \begin{equation}\label{rKVp}
	      K=V'(u_M)=V'(u_m)
	     \end{equation}
	     hold since $u_M$ and $u_m$ are critical points 
	     of the function $K u - V (u)$.
	     These with \eqref{luacm} and \eqref{lubcm}
	     show that $f' (u_0) (= - \lim_{q \to u_0} \mu_{per})$ 
	     converges to $0$ as $c \to c_M, c_m$.
	     Therefore (\ref{hobp1}) and (\ref{hobp2}) imply that
	     a Bogdanov-Takens bifurcation (\cite{MR2071006})
	     occurs at $(c, \mu) = (c_M, 0), (c_m, 0)$.
	     We emphasize that the propositions yield information 
	     on not a local bifurcation diagram but a global one;
	     the proofs will be done without using the theory of local bifurcations.
 \end{itemize}
\end{remark}


\subsection{Proof of Proposition~\ref{prop:hobp}}

In the following proofs of this subsection, we ignore the dependence on $K$ 
in order to simplify notation.
Before the proof of Proposition~\ref{prop:hobp}, 
we examine the behavior of $u^\pm_1$ and $w^\pm_1$ 
(resp. $u^\pm_2$ and $w^\pm_2$) as $c \to c_M$ (resp. $c \to c_m$).


\begin{lemma}\label{lem:bucm}
 Give $\mu^+_\infty \in [-\infty,0]$ and $\mu^-_\infty \in [0,\infty]$ arbitrarily.
 Fix $K > 0$.
 If $(c, \mu)$ converges to $(c_M, \mu^\pm_\infty)$ and satisfies 
 $(K, c) \in \mathcal{D}_1$, then $u^\pm_1 (K, c, \mu) \to u_M (K)$ 
 and $w^\pm_1 (u; K, c, \mu) \to 0$ uniformly in $u$ under the limit.
 Similarly, if $(K, c) \in \mathcal{D}_1 \cup \mathcal{D}_2$ 
 and $(c, \mu) \to (c_m, \mu^\pm_\infty)$, then $u^\pm_2 (K, c, \mu) \to u_m (K)$ 
 and $w^\pm_2 (u; K, c, \mu) \to 0$ uniformly in $u$.
\end{lemma}


\begin{proof}
 We may represent $w^+_1 (u; c, \mu)$ by $w^+_1 (u)$ for simplicity.
 In order to show the assertion for $u_1^+$ and $w_1^+$,
 it is sufficient to consider only the case $\mu^+_\infty = 0$
 because it follows from Lemmas~\ref{lem:mwmu} and \ref{lem:mwmc} that 
 \[
 u_0 (c) \le u^+_1 (c,\mu) \le u^+_1 (c, 0),
 \quad 0 \le w^+_1(u; c,\mu) \le w^+_1 (u; c, 0)
 \]
 if $\mu \le 0$.
 Applying Lemma~\ref{lem:Wpes}
 with $W=w^+_1$, $A=2g_1 f $, $B=2g_2h(\cdot,\mu)$, $s_1=u_1$ and $s_2=u$ gives
 \begin{equation}\label{wpoea}
  w^+_1(u)^2 \le \int_{u_1}^{u} e^{u-s} 
   \left(2 g_1 (s) f (s) + g_2 (s)^2 h (s,\mu)^2 \right) ds
 \end{equation}
 for $u \in [u_1, u^+_1]$.
 From this and \eqref{luacm}, we particularly have
 \begin{equation}\label{wpol}
  w^+_1 (u_0; c, \mu) \to 0 \ \mbox{ as } \ (c,\mu) \to (c_M, 0).
 \end{equation}
 Integrating \eqref{eq:wueq} over $[u_0,u]$ 
 and using the fact that $f \le 0$ on $[u_0,u_2]$ yield
 \[
 w^+_1 (u) \le w^+_1 (u_0) +\int_{u_0}^{u} g_2 (s) |h (s,\mu)| ds
 \]
 for $u \in [u_0,u^+_1]$.
 Furthermore, integrating \eqref{eq:wueq2} over $[u_0,u]$ 
 and then plugging the above inequality into the result, we deduce that 
 \begin{equation}\label{wpoe}
  \begin{aligned}
   \frac{1}{2}w^+_1 (u)^2
   & = \frac{1}{2}w^+_1 (u_0)^2 
   + \int_{u_0}^{u} (g_1 (s) f (s) + g_2 (s) h (s,\mu) w^+_1 (s)) ds \\
   & \le \frac{1}{2}w^+_1 (u_0)^2 + \int_{u_0}^{u} g_1(s)f(s) ds
   + \int_{u_0}^{u} g_2(s)|h(s,\mu)|
   \left( w^+_1 (u_0) + \int_{u_0}^{s} g_2 (\tau) |h (\tau,\mu)| d\tau \right) ds   
  \end{aligned}
 \end{equation}
 for $u \in [u_0,u^+_1]$.

 Define $u^* \equiv \limsup_{(c, \mu) \to (c_M, 0)} u_1^+ (c, \mu)$.
 By \eqref{luacm} and \eqref{wpol},
 we see that the right-hand side of (\ref{wpoe}) converges to 
 \[
 \begin{aligned}
  I (u) 
  \equiv \int_{u_M}^{u} g_1(s) f(s; c_M) ds
  + \int_{u_M}^{u} g_2 (s) | h ( s, 0 ) |
  \left( \int_{u_M}^{s} g_2 (\tau) |h ( \tau, 0)| d \tau \right) ds   
 \end{aligned}
 \]
 for each $u \in [u_M, u^*]$ as $(c, \mu) \to (c_M, 0)$,
 where we used the notation $f (u; c)$ to emphasize $c$-dependency of $f$.
 Recall that $f (u_M; c_M) = f' (u_M; c_M) = 0$ 
 and $h (u, 0) = f' (u; c) < 0$ for $u \in (u_M, u_m)$.
 Then $I (u)$ is estimated as
 \begin{equation}\label{eq7_Iu}
   I(u) \le \int_{u_M}^{u} ( g_1 (s) + C f'(s; c_M) ) f (s;c_M) ds
 \end{equation}
 for $u \in [u_M, \min \{u^*, u_m\}]$, where $C > 0$ is some constant.
 Then $u^*$ must be equal to $u_M$.
 Otherwise, since $f'(u; c_M)$ is sufficiently small, the integral on 
 the right-hand side of (\ref{eq7_Iu}) is negative for $u$ close to $u_M$, 
 which contradicts the fact that the left-hand side of (\ref{wpoe}) is nonnegative.
 Therefore $u^+_1(c, \mu) \to u_M$ as $(c, \mu) \to (c_M, 0)$.
 From this and \eqref{wpoea}, we also obtain the uniform convergence of $w^+_1$ to $0$.

 The remainder of the lemma can be shown in the same way as above.
 So we omit the details of the proofs.
\end{proof}


\begin{proof}[Proof of Proposition~\ref{prop:hobp}]
 Recall that $\mu_f<\mu^1_{pul}<\mu_b$ and $\mu_b<\mu^2_{pul}<\mu_f$,
 which were shown in the proof of Theorem~\ref{thm:HOCO}.
 We hence have $\mu^j_{pul} \to \mu_*$ as $c \to c_*$.
 This with the continuous dependence of stable and unstable manifolds on parameters
 gives the convergence of $\mathcal{O}^j_{pul}$ to $\mathcal{O}_*$.
 We have therefore proved \eqref{hobp0}.

 Let us show \eqref{hobp1}.
 It suffices to verify that $\mu^1_{pul} \to 0$ as $c \to c_M$.
 Indeed, if this is true, it is shown that $\mathcal{O}^1_{pul} \to \{ (u_M, 0) \}$
 from Lemma~\ref{lem:bucm} and the fact that
 \[
 \begin{aligned}
  \mathcal{O}^1_{pul} 
  = \{ (u, w^+_1 (u; c,\mu^1_{pul})) \ | \ u \in (u_1, u^+_1] \}
  \cup \{ (u,w^-_1(u; c, \mu^1_{pul})) \ | \ u \in (u_1, u^-_1] \}.   
 \end{aligned}
 \]
 Let $\{ c_n \}_n$ be any sequence such that $\lim_{n \to \infty} c_n = c_M$ and
 \[
 \mu_\infty \equiv \lim_{n \to \infty} \mu^1_{pul} (c_n) \in [-\infty,\infty]
 \]
 exists.
 We prove $\mu_\infty = 0$.
 We apply Lemma~\ref{lem:bucm} for $\mu_\infty = \mu^+_\infty$ if $\mu_\infty \le 0$
 and for $\mu_\infty = \mu^-_\infty$ if $\mu_\infty \ge 0$.
 In either case, we have 
 \begin{equation}\label{mune}
  u^n_1 \equiv u^+_1 (c_n, \mu^1_{pul} (c_n)) = u^-_1 (c_n, \mu^1_{pul} (c_n)) \to u_M
 \end{equation}
 as $n \to \infty$.
 We now use the inequalities
 \[
 - \sup_{u \in (u_1,u^n_1)} f'(u)
 < \mu^1_{pul}(c_n) < - \inf_{u \in (u_1,u^n_1)} f'(u), 
 \]
 which follows from Lemma~\ref{lem:empl}.
 From \eqref{luacm}, \eqref{rKVp}, and \eqref{mune}, 
 we find that both the left-hand and the right-hand sides 
 converge to $0$ as $n \to \infty$, which leads to $\mu_\infty = 0$.
 We can derive \eqref{hobp2} in a similar way, and thus the proof is complete.
\end{proof}


\subsection{Proof of Proposition~\ref{prop:pobp}}

We define orbits $\mathcal{O}^{\pm}$ by
\[
\mathcal{O}^{\pm} \equiv 
\left\{
\begin{aligned}
 & \{ (u, w^{\pm}(u)) \ | \ u \in [q, u^{\pm}] \}
 & & \mbox{if } \ q \in (u_1,u_0), \\
 & \{ (u, w^{\pm}(u)) \ | \ u \in [u^{\pm}, q] \}
 & & \mbox{if } \ q \in (u_0,u_2).
\end{aligned}
\right.
\]
%


\begin{lemma}\label{lem:cwq}
 For $j=1, 2$, there hold $\mathcal{O}^{\pm} \to \overline{\mathcal{O}^{\pm}_j}$
 locally uniformly in $\mu \in \mathbb{R}$ as $q \to u_j$,
 where $\mathcal{O}^{\pm}_j$ were defined in (\ref{eq3_Opmj}).
 In particular, $u^{\pm} \to u^{\pm}_j$ as $q \to u_j$.
\end{lemma}


\begin{proof}
 The convergence of $\mathcal{O}^{\pm}$ to $\overline{\mathcal{O}^{\pm}_j}$ 
 in a neighborhood of the equilibrium $(u_j, 0)$ 
 follows from the Hartman-Grobman theorem.
 The proof for the convergence away from equilibria is then standard.
\end{proof}


Let us conclude this section by showing Proposition~\ref{prop:pobp}.


\begin{proof}[Proof of Proposition~\ref{prop:pobp}]
 From (ii) of Lemma~\ref{lem:evus}, we see that no periodic orbit exists 
 in a neighborhood of the equilibrium $(u_0,0)$ provided $\mu \neq - f'(u_0)$.
 Hence $\mu_{per}$ must converge to $- f'(u_0)$ as $q \to u_0$.
 By the same argument as in the proof of Theorem~\ref{thm:HOCO} (ii),
 the behavior of $(u_{per}, w_{per})$ is approximated by 
 \[
 (u_0, 0) + (q - u_0) (\cos \omega_0 z, - \omega_0 \sin \omega_0 z) 
 \]
 uniformly in $z$ if $q$ is close to $u_0$.
 This implies that $(u_{per}, w_{per}) \to (u_0,0)$ 
 and $Z_{per} \to 2 \pi / \omega_0$ as $q \to u_0$ from the representation above.
 Therefore \eqref{pobp1} holds.

 Let us proceed to the proof of \eqref{pobp21}.
 We consider the case $(K, c) \in \mathcal{D}_{1,1}$.
 To emphasize the dependency of $q$, 
 we may write $\mu_{per} = \mu_{per}^q$ and $u^{\pm} = u^{q, \pm}$.
 On the other hand, we ignore the dependence of $u^\pm_1$ and $u^{q, \pm}$ 
 on $(K,c)$ and write $u^\pm_1 (\mu)$ and $u^{q, \pm} (\mu)$ 
 instead of $u^\pm_1 (K, c, \mu)$ and $u^{q, \pm} (K, c, \mu)$ for simplicity.
 We show that $\mu_{per}^q \to \mu^1_{pul}$ as $q \to u_1$ by contradiction.
 If this is false, we can take a constant $\e_0 > 0$ 
 and a sequence $\{ q_n \}_n$ which satisfies
 $q_n \to u_1$ as $n \to \infty$ and 
 $|\mu_{per}^{q_n} - \mu^1_{pul}| \ge \e_0$ for all $n$.
 Let us consider the case that $\mu_{per}^{q_n} \ge \mu^1_{pul} + \e_0$ 
 for infinitely many $n$.
 By (\ref{eq6_lem:prwq}), we have
 \[
 u^{q_n, +} (\mu^1_{pul}+\e_0) \le u^{q_n, +} (\mu_{per}^{q_n})
 = u^{q_n, -} (\mu_{per}^{q_n}) \le u^{q_n, -} (\mu^1_{pul} + \e_0). 
 \]
 Hence letting $n \to \infty$ and using Lemma~\ref{lem:cwq} yield
 $u^+_1 (\mu^1_{pul} + \e_0) \le u^-_1 (\mu^1_{pul} + \e_0)$.
 However it follows from Lemma~\ref{lem:mwmc} that 
 \[
 u^+_1 (\mu^1_{pul} + \e_0) > u^+_1 (\mu^1_{pul})
 = u^-_1 (\mu^1_{pul}) > u^-_1(\mu^1_{pul} + \e_0),
 \]
 which is a contradiction.
 We can deal with the other case in the same way as above.

 The convergence of $\mathcal{O}_{per}$ to $\mathcal{O}^1_{pul}$
 is verified by combining Lemma~\ref{lem:cwq} 
 and the fact that $\mu_{per}^q \to \mu^1_{pul}$ as $q \to u_1$.
 Furthermore, since the initial value $(u_{per} (0), w_{per} (0)) = (q, 0)$ 
 approaches the equilibrium $(u_1, 0)$,
 it is easy to see that $Z_{per}$ diverges to $\infty$.
 We have thus proved \eqref{pobp21}, and the proof is complete.
\end{proof}


\section{Numerical continuation of traveling wave solutions}
\label{Numerical results}

We illustrate all branches of the traveling wave solutions in (\ref{eq1_TV})
with the constants $V_0 = 0.0168, M = 0.913, u_c = 0.025, \beta = 89.7$, 
which are identical to those in Figure~\ref{fig_congestion}.
These constants are fixed throughout this section.
We used the numerical continuation package HomCont/AUTO \cite{doedel2007auto} 
for heteroclinic, homoclinic, and periodic orbits. 
The numerical approximations of the heteroclinic and homoclinic orbits were 
achieved by solving a truncated problem using the projection boundary conditions. 
Refer to \cite{doedel1990numerical}, \cite{friedman1991numerical}, and
\cite{beyn1990numerical} for the theoretical background.


\begin{figure*}[h]
 \begin{center}
  \includegraphics[width=14cm]{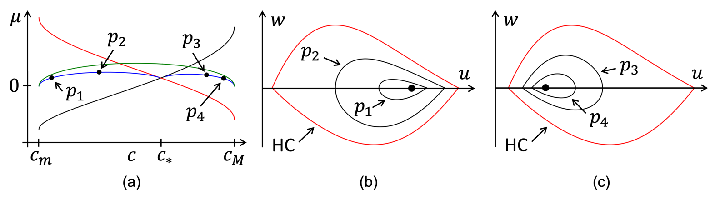}
 \end{center}
 \caption{Numerical continuations for heteroclinic and homoclinic orbits,
 and Hopf bifurcation points. 
 (a) Branches of traveling back (black), front (red) and pulse (blue) solutions.
 The green curve represents the graph of $\mu = - f' (u_0)$.
 (ii) of Lemma~\ref{lem:evus} shows
 that a Hopf bifurcation occurs in (\ref{eq1_TV}) at each point of the green curve.
 All the branches terminate at the values of $c = c_m$ or $c = c_M$ 
 where two of the equilibriums among the three collide. 
 (b), (c) Corresponding orbits for traveling pulse solutions 
 at $p_1, p_2, p_3, p_4$ on the blue branches in (a).
 They are homoclinic orbits in the $(u, w)$ phase plane, 
 which are represented by black curves.
 Red curves indicate the heteroclinic cycle (HC), 
 consisting of two heteroclinic orbits 
 at the intersection $(c_*, \mu_*)$ of the black and red curves in (a).
 The black disks in (b), (c) are $(u_m, 0)$ and $(u_M, 0)$
 corresponding to the endpoints $(c_m, 0)$ and $(c_M, 0)$ 
 of the blue and green curves in (a), respectively.
 }
 \label{c-mu}
\end{figure*}


We study the model proposed by Lee et al., 
characterized by the function $\k (\rho) = 1 / (6 \tau \rho^2)$.
Notably, the system described by (\ref{eq1_TV}) remains independent of $\tau$.
We fixed $K = 1.25$,
resulting in approximate values of $c_1 = c_m \approx 0.0151$ and $c_M \approx 0.0167$.
Therefore, Theorem~\ref{thm:HECO} stipulates 
the existence of traveling back and front solutions
when $c_m < c < c_M$ and $\mu = \mu_b, \mu_f$.
Furthermore, traveling pulse solutions emerge when $\mu = \mu^1_{pul}, \mu^2_{pul}$.

Figure~\ref{c-mu} illustrates the branches of heteroclinic 
and homoclinic orbits within the $(c, \mu)$-parameter space, 
obtained through numerical continuations. 
Each point $(c,\mu)$ along the branches (black, red, or blue) 
in the upper-left plot corresponds to specific parameter settings conducive 
to heteroclinic or homoclinic orbits.
These branches represent the loci of $\mu = \mu_b, \mu_f, \mu^1_{pul}, \mu^2_{pul}$.
Notably, the crossing of two branches of heteroclinic orbits occurs 
at the point $(c_*, \mu_*)$, indicating the presence of a heteroclinic cycle. 
Additionally, the two branches of the homoclinic orbits diverge from this point. 
These findings are consistent with 
Theorems~\ref{thm:HECO} and \ref{thm:HOCO}, and Proposition~\ref{prop:hobp}.

The bifurcation of the homoclinic orbit from the heteroclinic cycle was discussed
in \cite{kokubu1988homoclinic}.
While this previous study necessitated non-degeneracy hypotheses, 
we do not undertake such analytical investigations. 
However, numerical estimates of saddle quantity, computed as 
the sum of eigenvalues at saddle points 
$O_1 = (u_1, 0)$ and $O_2 = (u_2, 0)$ in (\ref{eq1_TV}), were performed.
We estimated $u_1 = 0.0166$ and $u_2 = 0.0344$,
along with the parameter values $(c_*, \mu_*) = (0.01611, 0.0734)$.
Eigenvalue calculations for $\l_\pm (u_i)$, 
as defined in Lemma~\ref{lem:evus}, yielded approximate values of 
$(\l_- (u_1), \l_+ (u_1)) = (-52.88, 117.82)$ 
and $(\l_- (u_2), \l_+ (u_2)) = (-27.90, 66.08)$. 
Consequently, positive saddle quantities were deduced at $(u_1, 0)$ and $(u_2, 0)$, 
signifying that the homoclinic orbit branches from $O_1$ to $O_1$ and 
from $O_2$ to $O_2$ are tangential to the heteroclinic orbit branches 
from $O_1$ to $O_2$ and from $O_2$ to $O_1$ at $(c^*,\mu^*)$, respectively.
This tangency is evident in Figure~\ref{c-mu} (a).


\begin{figure}[h]
 \begin{center}
  \includegraphics[width=8.5cm]{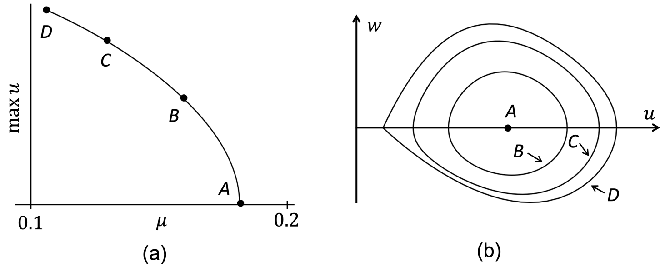} \\
  \includegraphics[width=8.5cm]{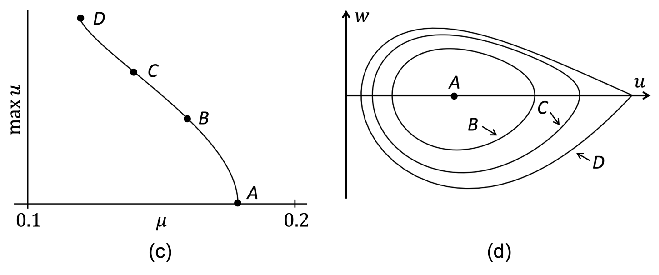}
 \end{center}
 \caption{Continuation of periodic orbits from the Hopf 
 to homoclinic bifurcation.
 (a), (c) Maximum of $u$ versus $\mu$. 
 (b), (d) Orbits in the $(u,w)$ phase plane.
 We set $c = 0.0163$ in the upper figures, 
 whereas $c = 0.0155$ in the lower ones.
 The equilibrium point $(u_0, 0)$ is denoted by $A$, while
 $B$ and $C$ represent periodic orbits, and $D$ corresponds to a homoclinic orbit.
 }
 \label{per00}
\end{figure}


The numerical continuation process relies on having approximate heteroclinic 
orbit as a starting point, which must be sufficiently accurate. 
While having an exact solution at a specific parameter value is advantageous, it is
not feasible for (\ref{eq1_TV}).
However, the Allen-Cahn-Nagumo equation
\begin{equation}\label{eq:ACN}
 \left\{
  \begin{aligned}
   & u_z = w, \\
   & w_z = - \mu w + u (u-a) (u-1)
  \end{aligned}
 \right.
\end{equation}
offer exact families of heteroclinic solutions 
\begin{equation}\label{exact_het}
 u (z) = \frac{1}{1 + e^{z/\sqrt{2}}},
  \quad \mu = \sqrt{2}\left( \frac{1}{2}-a \right)
\end{equation}
for $a \in (0, 1)$, and homoclinic solutions
\begin{equation}\label{exact_hom}
 u (z) = \frac{6a}{2(1+a) + \sqrt{2 (2-a)(1-2a)} \cosh \sqrt{a} z},
  \quad \mu = 0
\end{equation}
for $a \in (0, 1/2)$.
These exact solutions can serve as the seeds for homotopy continuation 
from (\ref{eq:ACN}) to (\ref{eq1_TV}) for $a \in (0, 1/2)$.
Specifically, we scaled the nonlinear terms $u (u-a) (u-1)$ in (\ref{eq:ACN}) linearly,
defining $f_{AC} (u; u_1, u_0, u_2) = (u-u_1) (u-u_0) (u-u_2)$, 
and performed homotopy continuation for
\begin{equation}\label{eq:homotopy}
 \left\{
  \begin{aligned}
   u_z & = w, \\
   w_z & = -\mu w+(1-\phi) f_{AC} (u; u_1, u_0, u_2)
   + \phi (g_1 (u) f (u) + g_2 (u) h (u,\mu) w),
  \end{aligned}
 \right.
\end{equation}
where $\phi \in [0,1]$ represents a homotopic parameter.
Note that $f_{AC} (u; u_1, u_0, u_2)$ and $f (u)$ have the same zeroes 
for $(K,c) \in \mathcal{D}_1$.
At $\phi = 0$, the exact solutions of the 
heteroclinic and homoclinic orbits are as described earlier.
Therefore, the continuation from $\phi = 0$ to $1$ yields
approximate solutions for (\ref{eq1_TV}).


\begin{figure}[h]
 \begin{center}
  \includegraphics[width=8.5cm]{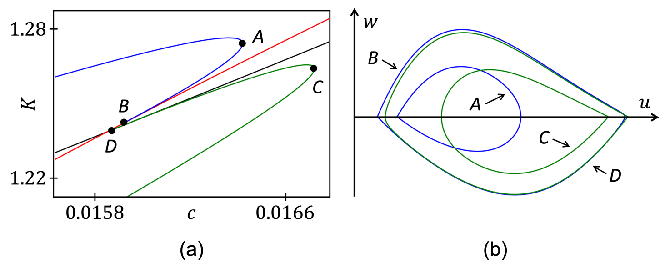}
 \end{center}
 \caption{(a) Continuation branches for the traveling back (black),
 traveling front (red), and traveling pulse (green, blue) solutions 
 in the $(K, c)$-parameter space.
 (b) Orbits in the $(u,w)$ phase plane.
 The curves $A, B$ (blue) and $C, D$ (green) correspond to
 the points along the branches of traveling pulse solutions in (a).}
 \label{c-K}
\end{figure}


We delve into the bifurcation of periodic orbits, as outlined
in Proposition~\ref{prop:pobp}, using the same parameters depicted in Figure~\ref{c-mu}.
$\mu = - f' (u_0)$ serves as a critical value for each $c \in (c_m, c_M)$,
where the middle equilibrium $(u_0, 0)$ possesses purely imaginary eigenvalues,
as denoted by the green curve in Figure~\ref{c-mu}.
We can numerically trace the continuation of periodic orbits from the Hopf
bifurcation point by considering $\mu$ as a control parameter while 
$c$ remains fixed in $(c_m, c_M)$. 
The branches of these periodic orbits seem to culminate in the homoclinic orbit,
represented by the blue curves in Figure~\ref{c-mu} (a).
Figure~\ref{per00} showcases two typical situations: 
$(K, c) \in \mathcal{D}_{1,1}$ ((a), (b))
and $(K, c) \in \mathcal{D}_{1,2}$ ((c), (d)).
We infer that periodic orbits exist within the parameter region 
delineated by the three curves in Figure~\ref{c-mu} (a): 
the two blue curves (branches of homoclinic orbits) 
and the green curve (Hopf bifurcation points).

Furthermore, we delineated the branches of heteroclinic and homoclinic orbits 
in the $(K, c)$-parameter space illustrated in Figure~\ref{c-K}, 
where we set $\mu = 2 \tau K - 1$.
As depicted in the figure, bifurcation branches 
for heteroclinic and homoclinic orbits correspond to the boundary conditions 
(\ref{eq:bcmi}), (\ref{eq:bcmd}), and (\ref{eq:bctmp0}), considering a given $u_1, u_2$.
Obtaining a figure akin to Figure~\ref{c-K} rigorously is challenging 
because $K$ and $c$ feature in various locations in (\ref{eq1_TV}). 
Hence, regarding $\mu$ as an independent parameter is apt. 
We could obtain qualitatively similar results for the K\"uhne model 
as shown in Figure~\ref{c-K} for Lee et al.'s model. 
However, we have not reported these findings.

We discuss the relationship between the traveling wave solutions for (\ref{eq1_TV}) 
and those of the original problem (\ref{eq_cOV}).
We numerically compute $K$ and $c$ using the data used from 
Figure~\ref{fig_congestion} and confirm that the preceding discussion implies 
the existence of a traveling wave solution in (\ref{eq_cOV}).
First, we estimate the time required for the pulse to traverse the region as $T = 178$.
After that, $c = L/T$ is approximated to be $0.0130899$.
Subsequently, we determine $(\rho_-, v_-) = (26.516, 0.029025)$
using the data for $\rho (x, t)$ at $(x, t) = (0, 1000)$.
Our observations indicate that 
$x = 0$ lies outside the congestion phase at $t = 1000$.
We use (\ref{eq1_preserve_K}) to find $K = \rho_- (v_- + c) \approx 1.11672$. 
Using the estimated values of $(K, c)$, and employing AUTO like in the case 
from which we obtained the results depicted in Figure~\ref{per00}, 
we derive a periodic traveling wave solution with a period of $2.33022$,
which closely aligns with $L = 2.33$.
Moreover, we obtain $c = 0.013089412$ and $\mu_{per} = 0.11640802175$.
We note that $(K,c) \in \mathcal{D}_1$,
$\mu_{per}$ is approximately $2 \tau K - 1 \approx 0.11672$, and $(c, \mu_{per})$ 
is included in the parameter region bounded by the three curves 
related to the branches of homoclinic and Hopf bifurcation points, 
as shown in Figure~\ref{c-mu} (a).
Hence, we conclude that the solution illustrated in Figure~\ref{fig_congestion}
coincides with the periodic orbit verified numerically by AUTO.


\section{Discussions}
\label{Discussions}

This study rigorously established the existence 
of various traveling wave solutions in the macroscopic traffic model~(\ref{eq_cOV}). 
The emergence of congested states as time-periodic solutions in microscopic models 
is highlighted via Hopf bifurcation, which is a promising method 
for obtaining such solutions.
However, obtaining a solution away from the bifurcation point 
is often infeasible.
Alternatively, \cite{sugiyama1997simple} employed
the step function as an OV function and successfully formally constructed
a solution corresponding to the congestion phase.
Consequently, strong restrictions are usually necessary 
to rigorously obtain the congestion phase in microscopic models.
Therefore, continuous models are useful in treating traveling wave solutions 
in a congestion phase.

In presenting Theorems~\ref{thm:HOCO} and \ref{thm:ptws}, 
condition (\ref{cd:nu}) was considered.
If the viscosity coefficient $\k (\rho)$ exhibits a strong singularity at $\rho = 0$, 
as in Lee et al.'s model, all theorems in this study remain valid.
However, (\ref{cd:nu}) does not hold in the case where $\k (\rho)$ is constant, 
as in the K\"uhne model, or has a weak singularity, as in the Kerner 
and Konh{\"a}user model.
Actually, constructing a heteroclinic orbit is feasible
in (\ref{eq1_TV}) protruding outside the region $\{ u>0 \}$ even if $\k (\rho) = \k_0$.
However, this solution is meaningless in (\ref{eq_cOV}) 
because $\rho = 1 / u$ should be positive.
Further analyses are necessary to obtain analogous results 
to Proposition~\ref{prop:BMBF} without (\ref{cd:nu}).


\bibliographystyle{elsarticle-num} 
\bibliography{reference}


\end{document}